\documentclass[12pt]{amsart}

\title[The Quantum Variance of the Modular Surface] {The Quantum Variance of the Modular Surface \\
}

\author[\sc P. Sarnak and P. Zhao, Appendix by M. Woodbury]{\sc Peter Sarnak
and Peng Zhao
\\
Appendix by Michael Woodbury}
\date{Feb. 9, 2018}
\address
{Department of Mathematics, Princeton University and IAS,
Princeton, NJ 08540}

\email {sarnak@math.princeton.edu}

\address
{Department of Mathematics, Columbia University,
New York, NY 10027}

\email {woodbury@math.columbia.edu}

\address
{Department of Mathematics, University of Connecticut,
Storrs, CT 06269}

\email {peng.2.zhao@uconn.edu}

\usepackage{amsmath}
\usepackage{color}
\usepackage{longtable}
\usepackage[retainorgcmds]{IEEEtrantools} 
\usepackage{hyperref}
\usepackage{amssymb}
\newcommand{\sgn}{\operatorname{sgn}}
\newtheorem{theorem}{Theorem}
\newtheorem{corollary}{Corollary}
\newtheorem{lemma}{Lemma}

\newtheorem{proposition}{Proposition}
\newtheorem{remark}{Remark}
\newcommand{\W}{\mathcal{W}}
\newcommand{\DS}[1]{\pi_{{\rm dis}}^{#1}}
\newcommand{\GL}{\mathrm{GL}}
\newcommand{\SL}{\mathrm{SL}}
\DeclareMathOperator{\vol}{vol}
\newcommand{\PGL}{\mathrm{PGL}}
\newcommand{\SO}{\mathrm{SO}}
\DeclareMathOperator{\Ad}{Ad}
\DeclareMathOperator{\Ind}{Ind}
\newcommand{\abs}[1]{\left| #1 \right|}
\newcommand{\R}{\mathbb R}
\newcommand{\C}{\mathbb C}
\newcommand{\Z}{\mathbb Z}
\newcommand{\A}{\mathbb A}
\newcommand{\smatr}[4]{\left(\begin{smallmatrix} #1 & #2 \\ #3 & #4\end{smallmatrix}\right)}
\newcommand{\matr}[4]{\left(\begin{array}{cc} #1 & #2 \\ #3 & #4\end{array}\right)}
\newcommand{\bs}{\backslash}
\newcommand{\wt}{\widetilde}
\DeclareMathOperator{\re}{Re}

\begin{document}

\maketitle

\begin{abstract}
The variance of observables of quantum states of the Laplacian on the modular surface is calculated in the semiclassical limit. It is shown that this hermitian form is diagonalized by the irreducible representations of the modular quotient and on each of these it is equal to the classical variance of the geodesic flow after the insertion of a subtle arithmetical special value of the corresponding $L$-function.
\end{abstract}

\section{Introduction}

Let  $G=PSL(2,\mathbb{R})$,  $\Gamma=PSL(2,\mathbb{Z})$ and $X=\Gamma\backslash\mathbb{H}$ be the  modular
surface. $X$ is a hyperbolic surface of finite area and it has a large discrete spectrum for the Laplacian (see \cite{He} and \cite{Sarnak0}). The corresponding eigenfunctions can be diagonalized and we denote these Hecke-Maass forms by 
$\phi_j$, $j=1, 2,\cdots$. They are real valued and satisfy
\begin{eqnarray}
\Delta\phi_j+\lambda_j\phi_j=0,\quad T_n\phi_j=\lambda_{j}(n)\phi_j
\end{eqnarray}
and we normalize them by
\begin{eqnarray}
\int_X\phi_j(z)^2 dA(z)=1.
\end{eqnarray}
Here $dA$ is the normalized hyperbolic area form and write $\lambda_j=\frac14+t_j^2$. If $\lambda>0$ then it is known that such a $\phi$ is a cusp form \cite{He}. $\phi_j$ has a Fourier expansion,
\begin{eqnarray}
\phi_j(z)=\sum_{n\neq 0}\frac{c_j(|n|)}{\sqrt{|n|}}W_{0,it_j}(4\pi |n|y)e(nx),
\end{eqnarray}
where $W_{0,it_j}$ is the Whittaker function. $X$ carries a further symmetry induced by the orientation reversing isometry $z\rightarrow -\overline{z}$ of $\mathbb{H}$ and our $\phi_j$'s are either even or odd with respect to this symmetry $r$
\begin{eqnarray}
\phi_j(rz)=\epsilon_j\phi_j(z), \quad \quad \epsilon_j=\pm 1.
\end{eqnarray}
Correspondingly 
\begin{eqnarray}
c_j(n)=\epsilon_jc_j(-n).
\end{eqnarray}

The Iwasawa decomposition of $g\in G$ takes the form 
\begin{eqnarray}
g=n(x)a(y)k(\theta)
\end{eqnarray}
where
$$n(x)=\left(\begin{array}{cc} 1& x\\
0& 1\end{array}\right), a(y)=\left(\begin{array}{cc} y^{\frac12}& 0\\
0& y^{-\frac12}\end{array}\right), k(\theta)=\left(\begin{array}{cc} \cos\theta& \sin\theta\\
-\sin\theta& \cos\theta\end{array}\right).$$
$\mathbb{H}$ may be identified with $G/K$ where $K=SO(2)/(\pm I)$ and then 
$\Gamma\backslash G$ is identified with the unit tangent space or phase space for the geodesic flow on $X$. The objects whose fluctuations we study in this paper are the Wigner distributions $d\omega_j$ on $\Gamma\backslash G$. These are quadratic functionals of the $\phi_j$'s and are given by (see the recent paper \cite{AZ} for a detailed description of these distributions as well as their basic invariance properties),

\begin{eqnarray}d\omega_j=\phi_j(z)\sum_{k\in\mathbb{Z}}\phi_{j,k}(z)e^{-2ik\theta}d\omega\end{eqnarray}
where $$d\omega=\frac{dx dy}{y^2}\frac{d\theta}{2\pi}.$$

Here the $\phi_{j,k}$ are the shifted Maass cusp forms of weight $k$, normalized such that $\|\phi_{j,k}\|_2=1$ by raising and lowering operators, $E_+$ and $E_-$ respectively, where \cite{J94}
$$E_+=e^{-2i\theta}(2iy\frac{\partial}{\partial x}+2y\frac{\partial}{\partial y}+i\frac{\partial}{\partial \theta}),$$
$$E_-=e^{2i\theta}(2iy\frac{\partial}{\partial x}-2y\frac{\partial}{\partial y}+i\frac{\partial}{\partial \theta}).$$ They are eigenfunctions of the Casimir operator $\Omega$, which acts on $C^{\infty}(\Gamma\backslash G)$. 

The basic question concerning the $\omega_j$'s is their behavior in the semi-classical limit $t_j\rightarrow \infty$. Lindenstrauss \cite{Lindenstrauss} and Soundararajan \cite{sou} have shown that for an ``observable'' $\psi\in C(\Gamma\backslash G)$ 
\begin{eqnarray}
\omega_j(\psi)\rightarrow\frac{1}{\vol(\Gamma\backslash G)}\int_{\Gamma\backslash G}\psi(g)dg,\quad \mathrm{as} \quad j\rightarrow \infty
\end{eqnarray}
where $dg$ is normalized Haar measure (i.e. a probability measure), this is the so called ``QUE'' property. 

It is known after Watson \cite{Wat1} and Jakobson \cite{J94} that the generalized Lindel\"of Hypothesis implies that if 
\begin{eqnarray}
\int_{\Gamma\backslash G}\psi(g)dg=0
\end{eqnarray}
then, for $\epsilon>0$
\begin{eqnarray}
\omega_j(\psi)\ll_{\epsilon}t_j^{-\frac12+\epsilon}
\end{eqnarray}

For the rest of the paper we will assume that the mean value of $\psi$ is 0, i.e. (9) holds. The main result below is the determination of the quantum variance, namely the mean-square of the $\omega_j(\psi)$'s. These are computed for special observables (ones depending only on $z\in X$) in \cite{Luo2} where the $\phi_j$'s are replaced by holomorphic forms, and in \cite{Z} for the $\omega_j$'s at hand. The extension to the general observable that is carried out here is substantially more complicated and intricate. It comes with a reward in that the answer on the phase space is conceptually much more transparent and elegant.

The variance sums
\begin{eqnarray}
S_{\psi}(T):=\sum_{t_j\leq T}|\omega_j(\psi)|^2
\end{eqnarray}
were introduced by Zelditch who showed (in much greater generality) that $S_{\psi}(T)=O(\frac{T^2}{\log T})$ \cite{Ze}. Corresponding to (10) we expect that in our setting $S_{\psi}(T)$ will be at most $T^{1+\epsilon}$, since by Weyl's law \cite{Sel}, $\sum_{t_j\le T}1\sim \frac{T^2}{12}$.  To each $\phi_j$ is associated its standard $L$-function $L(s,\phi_j)$ as well as its symmetric-square $L$-function, $L(s, \mathrm{sym}^2\phi_j)$. These and the other $L$-functions $L(s,\pi)$ that arise below have analytic continuations to $\mathbb{C}$ with a functional equation relating $s$ to $1-s$. Our notation is that $L(s,\pi)$ is the finite part and $\Lambda(s.\pi)$ the completed $L$-function. While $L(1,\pi)$ is nonzero and depends mildly on $\pi$, $L(\frac12,\pi)$ is a very subtle and much studied arithmetical invariant. For technical as well as arithmetical reasons it is natural to include weights in the variance sums (11). The ``harmonic'' weights $L(1,\mathrm{sym}^2\phi_j)$ satisfy $$t_j^{-\epsilon}\ll_{\epsilon}L(1,\mathrm{sym}^2\phi_j)\ll_{\epsilon}t_j^{\epsilon},$$ for $\epsilon>0$ (\cite{hl}, \cite{Iw90}) and they have a limiting distribution (\cite{Luo6}). In the end we can remove these harmonic weights as we do in Section 5 but for now we include them.

\begin{theorem}
Denote by $A_0(\Gamma\backslash G)$ the space of smooth right $K$-finite functions on $\Gamma\backslash G$ which are of mean 0 and of rapid decay. There is a sesquilinear form $Q$ on $A_0(\Gamma\backslash G)\times A_0(\Gamma\backslash G)$ such that
\begin{eqnarray}
& &
\lim_{T\rightarrow\infty}\frac1T\sum_{t_j\leq T}L(1,\mathrm{sym}^2\phi_j)\omega_j(\psi_1)\overline{\omega}_j(\psi_2)= Q(\psi_1,\psi_2).
\end{eqnarray}
\end{theorem}
We call $Q$ the quantum variance. The proof of Theorem 1 proceeds by proving the existence of the limit which comes with  an explicit but formidable expression for $Q$ see (34) of section 2. It involves infinite sums over arithmetic-geometric terms (twisted Kloosterman sums) and it appears very difficult to read any properties of $Q$ directly from (34). For example even that $Q$ is not identically zero (which is the case so that the exponent of $T$ in the theorem is the correct one) is not clear. Using some apriori invariance properties of $Q$ as well as some others that are derived from special cases of general versions of the daunting expression (34) allows us to eventually
 diagonalize $Q$.

In order to describe the result we need some more notation. The fluctuations of an observable $\psi\in C_0(\Gamma\backslash G)$ under the classical motion $\mathcal{G}_t$ by geodesics was determined in  \cite{Ra} and \cite{c}, and it asserts that as $T$ goes to infinity,
\begin{eqnarray}
  \frac{1}{\sqrt{T}}\int_0^T \psi(\mathcal{G}_t(g))dt
\end{eqnarray}
 as a random variable on $ \Gamma\backslash G$ becomes Gaussian with mean zero and variance $V$ given by
\begin{eqnarray}
\quad\quad V(\psi_1,\psi_2)=\int_{-\infty}^{\infty}\int_{\Gamma\setminus
G}\psi_1\left(g\left(\begin{array}{cc} e^{\frac{t}{2}}
& 0\\
0& e^{-\frac{t}{2}}\end{array}\right)\right)\overline{\psi_2(g)}dg
dt.
\end{eqnarray}

Note that (14) converges due to  the rapid decay of correlations for the geodesic flow. The correspondence principle suggests, and it has been conjectured in \cite{ef}, that for chaotic systems such as the one at hand, the quantum fluctuations are also Gaussian with a variance which agrees with the classical one in (14).

The distributions $\omega_j$ enjoy some invariance properties that are inherited by $Q$ and which are critical for its determination. The first is that $\omega_j$ is asymptotically invariant under time reversal, see Section 3. 
If $w$ is the involution of $\Gamma\backslash G$ given by 
\begin{eqnarray}
\Gamma g\rightarrow \Gamma g\left(\begin{array}{cc} 0
& 1\\
-1& 0\end{array}\right)
\end{eqnarray}
then 
\begin{eqnarray}
Q(w\psi_1,\psi_2)=Q(\psi_1,w\psi_2)=Q(\psi_1,\psi_2)
\end{eqnarray}

The second symmetry is special to $X$ and follows from (4);
\begin{eqnarray}
r\omega_j=\omega_j,  \quad Q(r\psi_1,\psi_2)=Q(\psi_1,r\psi_2)=Q(\psi_1,\psi_2)
\end{eqnarray}
So if the quantum variance is to be compared with the classical variance then it should be to the symmetrized form 
 \begin{eqnarray}
 V^{\mathrm{sym}}(\psi_1,\psi_2):=V(\psi_1^{\mathrm{sym}},\psi_2^{\mathrm{sym}})
 \end{eqnarray}
where 
 \begin{eqnarray}
\psi^{\mathrm{sym}}:=\frac 14\sum_{h\in H}h\psi
 \end{eqnarray}
for $H=\{1,w,r,wr\}$.

These same symmetries arose in connection with the arithmetic measures on $\Gamma\backslash G$ studied in \cite{Luo0}. In fact the arithmetic variance $B$ introduced in that paper turns out as we will show, to be very close to our quantum variance $Q$. We employ freely some of the techniques and notations in \cite{Luo0}.

The classical variance $V$ is diagonalized by the decomposition of $L_{\mathrm{cusp}}^2(\Gamma\backslash G)$ into irreducible representations under right translations by $G$. For simplicity we will restrict ourselves to examining $Q$ on $L_{\mathrm{cusp}}^{2}(\Gamma\backslash G)$, the continuous spectrum can be investigated similarly. We have 
 \begin{eqnarray}
L_{\mathrm{cusp}}^2(\Gamma\backslash G)=\bigoplus_{j=1}^{\infty}W_{\pi_j},
\end{eqnarray}
where $W_{\pi_j}$'s are irreducible cuspidal automorphic representations, each also invariant under the Hecke algebra. The $\pi_j$'s come in two types, the discrete series $W_{\pi_j^k}$, $k$ even, $j=1,2,\cdots, d_k$, $d_k$ being the dimension of the space of holomorphic  and antiholomorphic forms of weight $k$, and the spherical representations $\pi_j^0$ (see \cite{Luo0}). Thus
\begin{eqnarray}
L_{cusp}^2(\Gamma\backslash G)&=&\sum_{j=1}^{\infty}W_{\pi_j^0}\oplus\sum_{k\geq12}\sum_{j=1}^{d_k}\left(W_{\pi_j^k}\oplus W_{\pi_j^{-k}}\right)\nonumber\\
&:=&\sum_{j=1}^{\infty}U_{\pi_j^0}\oplus\sum_{k\geq12}\sum_{j=1}^{d_k}U_{\pi_j^k}\label{decomp}
\end{eqnarray}
where $d_k$ is either $[k/12]$ or $[k/12]+1$ depending if $k/2=1$ mod 6 or not. 


We can finally state our main result,
\begin{theorem}
Both $V^{\mathrm{sym}}$ and $Q$ are diagonalized by the orthogonal decomposition (21) and on each summand $U_{\pi_j^k}$, we have
\begin{eqnarray}
Q|_{U_{\pi_j^k}}=L(\frac12,\pi_j^k)V^{\mathrm{sym}}|_{U_{\pi_j^k}}.
\end{eqnarray}
\end{theorem}

\begin{remark}
The precise meaning in Theorem 2 is that it holds when evaluated on any $\psi_1, \psi_2$ in $L_{cusp}^2(\Gamma\backslash G)\cap A_0(\Gamma\backslash G)$.
\end{remark}
\begin{remark}
The theorem asserts that the quantum variance is equal to the classical variance after inserting the ``correction factor" of $L(\frac12,\pi)$ on each irreducible subspace. As we have noted $Q$ is very close to the arithmetic variance $B$ in \cite{Luo0}. Comment (1.4.6) of that paper indicates heuristically why one might expect this to be so. However our proof that these Hermitian forms are essentially the same goes through a very different route.
\end{remark}
\begin{corollary} 
On removing the harmonic weights in (12) the resulting normalization constant in (22) for the variance is multiplied by a further positive number $C(\pi)$, which is a product of local densities;
$$C(\pi)=\frac{1}{\zeta(2)}\prod_p\left(1-\frac{\lambda_{\pi}(p)}{p^{3/2}(1+p^{-1})}\right)$$
where $\lambda_{\pi}(p)$ is the (normalized) eigenvalue of the Hecke operator $T_p$ on $\pi$.
\end{corollary}
We outline briefly the proofs of Theorem 1 and 2 and the contents of the  paper. Section 2 is devoted to the proof of Theorem 1. The variance sums are studied for functions in $A_0(\Gamma\backslash G)$, all of which are realized by Poincare series. 
The harmonic weight facilitates the use of the Petersson-Kuznetzov formula and the weights are only removed at the end. This technique was introduced in \cite{Luo7} and used in subsequent investigations \cite{J97}, \cite{Luo2} and \cite{Z} with progressively more complicated answers. The present case is given in Section 2 equation (34) and is (as we have noted) very complicated. We have to pass through versions of it as it is the only way that we know of proving the existence of the limit at this scale and we also need to use these formulae later to prove (23) below.

The rest of the paper, Sections 3 and 4 are concerned with diagonalizing $Q$. A key role is played by the asymptotic invariance of $\omega_j$ under the geodesic flow $\mathcal{G}_t$ on $\Gamma\backslash G$. This alone does not suffice to get the corresponding invariance property for $Q$, since we are working at the level slightly sharper than the bounds (10). To this end the recent results of Anantharaman and Zelditch \cite{AZ} clarify the exact error terms in the invariance properties of $\omega_j$ under $\mathcal{G}_t$. This together with well known multiplicity one results for linear functionals on irreducible representations of $G$, which are $\mathcal{G}_t$, $w$ and $r$ invariant, reduce the determination of $Q$ to $Q(\xi, \eta)$, where $\xi$ and $\eta$ are vectors which generate the irreducible $\pi_j^k$ and $\pi_{j'}^{k'}$ respectively (see \cite{Luo0}). If $\pi_j^k\neq\pi_{j'}^{k'}$, we need to show that $Q(\xi,\eta)=0$. This is done by establishing a self-adjointness property of $Q$ with respect to the finite Hecke operators $T_p$. Namely that for such $\xi$ and  $\eta$,
\begin{eqnarray}
Q(T_p\xi,\eta)=Q(\xi,T_p\eta)
\end{eqnarray}
The proof of this is given in Propositions 4 and 5 and requires one to prove several of identities for the corresponding twisted Kloosterman sums. This is similar to the analysis in applications of the trace formula to prove spectral identities, after comparisons of orbital integrals (the fundamental lemma as it is known in general). With (23) the vanishing of $Q(\xi,\eta)$, when $\pi_j^k\neq\pi_{j'}^{k'}$ follows from the multiplicity one theorem for automorphic cusp forms on $GL_2$. Finally when $\pi_j^k=\pi_{j'}^{k'}$ the sum (12) may be analyzed using Watson's triple product formula \cite{Wat1} and its generalization by Ichino \cite{ichino} together with techniques from averaging special values of $L$-functions over families. One needs an explicit form of these triple product identities for forms which are ramified at infinity. This is provided in Appendix A.  This  leads to the explicit evaluation of $Q(\xi,\eta)$, and in particular it introduces the magic factor of $L(\frac 12, \pi)$. Finally in Section 5, we remove the  harmonic weights and derive Corollary 1.

\section{Poincar\'{e} Series}
In this section we calculate the quantum variance sum of the weight $2k$
incomplete Poincar\'{e} series against $d\omega_j$ on $\Gamma\backslash G$.

 Let $h(t)$ be a smooth function on $(0,\infty)$ with compact support. 
On $C^{\infty}(0,\infty)$, define $\|\cdot\|_A$ by 
$$\| h\|_A=\max_{\substack{0\le i\le A , t\in(0,\infty)\\-A\le j\le A}}\Big|\frac{h^i(t)}{t^j}\Big|$$
For $m\in\mathbb{Z}$, define the incomplete Poincar\'{e} series of weight $-2k$:
$$P_{h,m,2k}(z,\theta)=e^{2ik\theta}\sum_{\gamma\in\Gamma_\infty\backslash\Gamma}h(y(\gamma z))(\epsilon_{\gamma}(z))^{2k}e(mx(\gamma
z)),$$
where \(\displaystyle\epsilon_{\gamma}(z)=\frac{cz+d}{|cz+d|}\) for \(\displaystyle\gamma=\left(\begin{array}{cc} *& *\\
c& d\end{array}\right)\).
For $m=0$, it becomes the incomplete Eisenstein series of the same weight.

On $\Gamma\backslash G$, define the Wigner distributioon $$d\omega_j=\varphi_j(z)\sum_{k\in\mathbb{Z}}\varphi_{j,k}(z)e^{-2ik\theta}d\omega$$
where $$d\omega=\frac{dx dy}{y^2}\frac{d\theta}{2\pi}.$$
 $\varphi_j$ is the $j$-th Hecke-Maass eigenform with the
corresponding Laplacian eigenvalue \(\displaystyle\lambda_j=\frac{1}{4}+t_j^2\),
Hecke eigenvalues $\lambda_j(n)$ and we normalize
$\|\varphi_j\|_2=1$. $\varphi_{j,k}(z)$ are shifted Maass cusp forms of weight $2k$, $\varphi_{j,k}(z)e^{-2ik\theta}$ is an eigenfunction of Casimir operator 
$$\Omega=y^2\left(\frac{\partial^2}{\partial
x^2}+\frac{\partial^2}{\partial y^2}\right)+y\frac{\partial^2}{\partial x\partial \theta}=\Delta+y\frac{\partial^2}{\partial x\partial \theta}$$
with the same eigenvalue \(\displaystyle \frac 14+t_j^2\) for every $k$. ($\Omega$ acts as $\Delta_{2k}=\Delta-2iky\frac{\partial}{\partial x}$ on weight $2k$ forms.)

We fix an even function $u(t)$ be analytic in the strip $|\mathrm{Im}t|<\frac12$ and  real analytic  on $\mathbb{R}$
satisfying $u^{(n)}(t)\ll(1+|t|)^{-N}$ for any $n>0$ and  large
$N$, and $u(t)\ll t^{N}$ when $t\rightarrow0$, for arbitrarily large $N$. And we assume $\int_{\mathbb{R}}u(t) dt=1.$

We have the following
\begin{proposition} \label{prop:1} For $h_1$, $h_2\in C_c^{\infty}(0,\infty)$, $m_1, m_2, k_1, k_2\in\mathbb{Z}$, and $P_{h_1,m_1,2k_1}$, $P_{h_2,m_2,2k_2}$ satisfying (9), there is a sesquilinear form $Q$ as in Theorem 1, such that 
\begin{eqnarray}
&&\lim_{T\rightarrow\infty} \frac1T\sum_{j\ge 1}u\left(\frac{t_j}{T}\right)L (1,\mathrm{sym}^2\varphi_j)\omega_j(P_{h_1,m_1,2k_1})\overline{\omega}_j(P_{h_2,m_2,2k_2})
\nonumber\\
&=& Q(P_{h_1,m_1,2k_1},P_{h_2,m_2,2k_2}).\nonumber
\end{eqnarray}

Moreover, there is a constant A and C (depending on $k_1$, $k_2$) such that the sesquilinear form Q satisfies
$$|Q(P_{h_1,m_1,2k_1},P_{h_2,m_2,2k_2})|\le C((|m_1|+1)(|m_2|+1))^A\|h_1\|_A\|h_2\|_A.$$
\end{proposition}


\begin{proof}
We prove the proposition for weight $-2k$, $k>0$ and it is analogous for functions of weight $2k$ (the case of $k_1=k_2=0$ being dealt with in \cite{Z}). Let $m_1m_2\neq 0$, without loss of generality, we  assume $m_1, m_2\in \mathbb{N}$. By the Iwasawa decomposition and unfolding we have
\begin{eqnarray} 
\omega_j(P_{h,m,2k}) & = &
\int_{\Gamma\backslash
G}(e^{2ik\theta}\sum_{\gamma\in\Gamma_\infty\backslash\Gamma}h(y(\gamma z))(\epsilon_{\gamma}(z))^{2k}e(mx(\gamma
z)))d\omega_j \nonumber \\
&=& \int_{\Gamma_{\infty}\backslash
\mathbb{H}}h(y)e(mx)\varphi_j(z)\varphi_{j,k}(z)d\mu(z)
  \end{eqnarray}
Apply the Fourier expansion of  $\varphi_{j,k}(z)$ \cite{J94},
$$
\varphi_{j,k}(z)=(-1)^k\Gamma(1/2+it_j)\sum_{n\neq 0}\frac{c_j(|n|)W_{\sgn(n)k,it_j}(4\pi|n|y)e(nx)}{\sqrt{|n|}\Gamma(\frac12+\sgn(n)k+it_j)},
$$
and $$\varphi_j(z)=\sum_{n\neq 0}\frac{c_j(|n|)W_{0,it_j}(4\pi|n|y)e(nx)}{\sqrt{|n|}\Gamma(\frac12+it_j)}.$$
From the relation $c_j(n)=c_j(1)\lambda_j(n)$ and the well-known multiplicativity of  Hecke eigenvalues
$$\lambda_j(n)\lambda_j(m)=\sum_{d|(n,m)}\lambda_j\left(\frac{mn}{d^2}\right),$$
we have
\begin{eqnarray}
\omega_j(P_{h,m,2k})  &=&4\pi(-1)^k\Gamma(\frac12+it_j)c_j(1)\sum_{d|m}\sum_{q\neq0,-\frac md}\frac{c_j(q^2+\frac{qm}{d})}{\sqrt{|1+\frac{m}{qd}|}}\nonumber\\
 &&\int_0^{\infty}\frac{W_{\sgn(q)k,it_j}(y)}{\Gamma(\frac12+\sgn(q)k+it_j)}W_{0,it_j}\left(y\Big|1+\frac{m}{qd}\Big|\right)h\left(\frac{y}{4\pi|qd|}\right)\frac{dy}{y^2}.\label{4}
\end{eqnarray}
 Let $H(s)$ be the Mellin transform of $h(y)$,
 $$H(s)=\int_0^{\infty}h(y)y^{-s}\frac{dy}{y}.$$
By  the Mellin inversion,
$$h(y)=\frac{1}{2\pi i}\int_{\sigma-i\infty}^{\sigma+i\infty}H(s)y^s ds,$$
for $\sigma>1$, the inner integral (25) can be written as 
$$\frac{1}{2\pi i}\int_{\sigma-i\infty}^{\sigma+i\infty}\frac{H(s)}{|4\pi qd|^s}\int_0^{\infty}y^{s-2}\frac{W_{\sgn(q)k,it_j}(y)}{\Gamma(\frac12+\sgn(q)k+it_j)}W_{0,it_j}\left(y\Big|1+\frac{m}{qd}\Big|\right)dy ds$$
Since $W_{0,\mu}(y)=\sqrt{y/\pi}K_{\mu}(y/2)$, we can denote the inner integral as
$$A_k(s)=\int_0^{\infty}y^{s-\frac 32}W_{\sgn(q)k,it_j}(2y)K_{it_j}\left(y\Big|1+\frac{m}{qd}\Big|\right)dy$$
 When $k=0$, the integral involves a product of two $K$-Bessel functions, which was evaluated by Luo-Sarnak \cite{Luo7}. Jakobson  \cite{J97} evaluated $A_1(s)$ using the standard properties of $K$-Bessel and Whittaker functions, 
 $$W_{1,it_j}=\sqrt{\frac 2\pi}(y^{\frac32}K_{it_j}(y)-y^{\frac12}(\frac12+it_j)K_{it_j}(y)+y^{\frac32}K_{it_j+1}(y))$$
 in which one gets $$A_1(s)=A_0(s+1)-(\frac12+it_j)A_0(s)+\sqrt{\frac{2}{\pi}}B(s)$$
where $$B(s)=\int_0^{\infty}y^{s}K_{it_j+1}(y)K_{it_j}\left(y\Big|1+\frac{m}{qd}\Big|\right)dy$$
Hence, 
\begin{eqnarray}
\quad\quad\quad\sqrt{\frac{\pi}{2}}A_1(s)&=&2^{s-2}\Gamma\left(\frac{s+1+2it_j}{2}\right)\Gamma\left(\frac{s+1-2it_j}{2}\right)|1+\frac{m}{qd}|^{it_j}\nonumber\\
&& \int_0^1\tau^{\frac{s-1}{2}}(1-\tau)^{\frac{s-1}{2}}(1+\frac{2\tau m}{qd}+\tau(\frac{m}{qd})^2)^{-\frac{s+1}{2}-it_j} d\tau\nonumber\\
&&-(\frac12+it_j)2^{s-3}\Gamma\left(\frac{s+2it_j}{2}\right)\Gamma\left(\frac{s-2it_j}{2}\right)|1+\frac{m}{qd}|^{it_j}\nonumber\\
&& \int_0^1\tau^{\frac{s-2}{2}}(1-\tau)^{\frac{s-2}{2}}(1+\frac{2\tau m}{qd}+\tau(\frac{m}{qd})^2)^{-\frac{s}{2}-it_j} d\tau\nonumber\\
&&+2^{s-2}\Gamma\left(\frac{s+2+2it_j}{2}\right)\Gamma\left(\frac{s-2it_j}{2}\right)|1+\frac{m}{qd}|^{it_j}\nonumber\\
&& \int_0^1\tau^{\frac{s-2}{2}}(1-\tau)^{\frac{s}{2}}(1+\frac{2\tau m}{qd}+\tau(\frac{m}{qd})^2)^{-\frac{s}{2}-1-it_j} d\tau
\end{eqnarray}
Similarly, we can obtain $A_{-1}(s)$ by the formula
$$A_{-1}(s)=\frac{A_0(s+1)}{\frac14+t_j^2}+\frac{A_0(s)}{\frac12-it_j}-\sqrt{\frac{2}{\pi}}\frac{B(s)}{\frac14+t_j^2}.$$

Then plug $A_1(s)$ and $A_{-1}(s)$ into (25) and by Stirling formula, Mellin inversion and the fact that \cite{Luo6} $$|c_j(1)|^2=\frac{2\cosh \pi t_j}{L(1,\textrm{sym}^2\varphi_j)},$$ 
we have
 \begin{eqnarray}
 \omega_j(P_{h,m,2})  &=&\frac{1}{L(1,\textrm{sym}^2\varphi_j)}\sum_{d|m}\sum_{q>0}\lambda_j(q^2+\frac{qm}{d})\tilde{H}(t_j,d,q,m)+O(t_j^{-2+\epsilon})\nonumber
 \end{eqnarray}
 

where
$$\tilde{H}(t_j,d,q,m)=\tilde{H}_1(t_j,d,q,m)+\tilde{H}_2(t_j,d,q,m)+\tilde{H}_3(t_j,d,q,m)$$ and
\begin{eqnarray}
\tilde{H}_1(t_j,d,q,m)&=&
\int_0^1 \frac{-(2\pi)^{3/2}(1+\frac{m}{qd})^{it_j-\frac12}}{(1+\frac{2\tau m}{qd}+\tau(\frac{m}{qd})^2)^{it_j}} (\tau(1-\tau)(1+\frac{2\tau m}{qd}+\tau(\frac{m}{qd})^2))^{-\frac 12}\nonumber\\
&& h\left(\frac{t_j\sqrt{\tau(1-\tau)}}{2\pi dq\sqrt{1+\frac{2\tau m}{qd}+\frac{\tau m^2}{(qd)^2}}}\right)d\tau,
 \end{eqnarray}
 
  \begin{eqnarray}
\tilde{H}_2(t_j,d,q,m)&=&
2\int_0^1 \frac{(2\pi)^{3/2}(1+\frac{m}{qd})^{it_j-\frac12}}{(1+\frac{2\tau m}{qd}+\tau(\frac{m}{qd})^2)^{it_j}}(\tau(1-\tau))^{-1}\nonumber\\
 && h\left(\frac{t_j\sqrt{\tau(1-\tau)}}{2\pi dq\sqrt{1+\frac{2\tau m}{qd}+\frac{\tau m^2}{(qd)^2}}}\right)d\tau,
 \end{eqnarray}
 and
  \begin{eqnarray}
\tilde{H}_3(t_j,d,q,m)&=&
\int_0^1 \frac{-(2\pi)^{3/2}(1+\frac{m}{qd})^{it_j-\frac12}}{(1+\frac{2\tau m}{qd}+\tau(\frac{m}{qd})^2)^{it_j}}(\tau(1+\frac{2\tau m}{qd}+\tau(\frac{m}{qd})^2))^{-1}\nonumber\\
 &&\quad\quad \quad\quad h\left(\frac{t_j\sqrt{\tau(1-\tau)}}{2\pi dq\sqrt{1+\frac{2\tau m}{qd}+\frac{\tau m^2}{(qd)^2}}}\right)d\tau,
 \end{eqnarray}

 For $i=1, 2$, we denote 
 \begin{eqnarray}
\omega_j(P_{h,m_i,2})  &=&\frac{1}{L(1,\mathrm{sym}^2\varphi_j)}\sum_{d_i|m_i}\sum_{q_i>0}\lambda_j(q_i^2+\frac{q_im_i}{d_i})(\tilde{H}_{1}(t_j,d_i,q_i,m_i)\nonumber\\
 &&\quad\quad\quad+\tilde{H}_{2}(t_j,d_i,q_i,m_i)+\tilde{H}_{3}(t_j,d_i,q_i,m_i)).
 \end{eqnarray}
 
Now, plug into
$$\sum_{j\ge 1}u\left(\frac{t_j}{T}\right)L (1,\mathrm{sym}^2\varphi_j)\omega_j(P_{h_1,m_1,2})\overline{\omega}_j(P_{h_2,m_2,2})$$

and apply Kuznetsov's formula \cite{Ku} to the inner sum, we obtain
\begin{eqnarray}
& &\sum_{j\ge 1}\lambda_j(q_1(q_1+\frac{m_1}{d_1}))\overline{\lambda_j(q_2(q_2+\frac{m_2}{d_2}))}\frac{1}{L(1,\mathrm{sym}^2\varphi_j)}\widetilde{h}(t_j)\nonumber\\
&=&\frac{\delta_{q_1(q_1+\frac{m_1}{d_1}),q_2(q_2+\frac{m_2}{d_2})}}{\pi^2}\int_{-\infty}^{\infty}t\tanh(\pi
t)\widetilde{h}(t)dt-\frac{2}{\pi}\int_{0}^{\infty}\frac{\widetilde{h}(t)d_{it}(q_1^2+q_1m_1/d_1)}{|\zeta(1+2it)|^2}\nonumber\\
& &d_{it}(q_2^2+q_2m_2/d_2)dt+\frac{2i}{\pi}\sum_c
c^{-1}S(q_1^2+q_1m_1/d_1,q_2^2+q_2m_2/d_2;c)\nonumber\\
&
&\int_{-\infty}^{\infty}J_{2it}(\frac{4\pi\sqrt{(q_1^2+q_1m_1/d_1)(q_2^2+q_2m_2/d_2)}}{c})t\frac{\widetilde{h}(t)}{\cosh(\pi
t)}dt.\nonumber
\end{eqnarray}
Here $$S(m,n;c)=\sum_{ad\equiv 1\mod c}e(\frac{dm+an}{c})$$ is the
Kloosterman sum and
$$d_{it}(n)=\sum_{d_1d_2=n}\left(\frac{d_1}{d_2}\right)^{it}.$$
and $$\tilde{h}(t)=\frac{1}{t^2}\tilde{H}(t,d_1q_1,m_1)\overline{\tilde{H}(t,d_2q_2,m_2)}u\left(\frac{t}{T}\right).$$

Thus, we have
\begin{eqnarray}
&
&\sum_{j\ge 1}u\left(\frac{t_j}{T}\right)L (1,\mathrm{sym}^2\varphi_j)\omega_j(P_{h_1,m_1,2})\overline{\omega}_j(P_{h_2,m_2,2})\nonumber\\
&=&
\frac{\pi^2}{32}\sum_{d_1,d_2,q_1,q_2}\Big(\frac{\delta_{q_1(q_1+\frac{m_1}{d_1}),q_2(q_2+\frac{m_2}{d_2})}}{\pi^2}\int_{-\infty}^{\infty}t\tanh(\pi
t)\widetilde{h}(t)dt\nonumber\\
&
&-\frac{2}{\pi}\int_{0}^{\infty}\frac{\widetilde{h}(t)}{|\zeta(1+2it)|^2}d_{it}(q_1^2+q_1m_1/d_1)d_{it}(q_2^2+q_2m_2/d_2)dt\\
& &+\frac{2i}{\pi}\sum_c
c^{-1}S(q_1^2+q_1m_1/d_1,q_2^2+q_2m_2/d_2;c) \nonumber\\& &\quad
\quad
\quad\int_{-\infty}^{\infty}J_{2it}(\frac{4\pi\sqrt{(q_1^2+q_1m_1/d_1)(q_2^2+q_2m_2/d_2)}}{c})t\frac{\widetilde{h}(t)}{\cosh(\pi
t)}dt\Big) ,
\end{eqnarray}
Next, we will estimate each of these  terms respectively.

First, we treat the diagonal terms. Since   for fixed $m_1$, $m_2,$ $q_1(q_1+\frac{m_1}{d_1})=q_2(q_2+\frac{m_2}{d_2})$ has 
 a uniformly bounded number of solutions  if $m_1/d_1\neq m_2/d_2$, and the
integer solutions to
$q_1(q_1+\frac{m_1}{d_1})=q_2(q_2+\frac{m_2}{d_2})$ are only
$q_1=q_2$ if $\frac{m_1}{d_1}=\frac{m_2}{d_2}$. Thus, the diagonal
terms are
$$\int_{-\infty}^{\infty} \frac {1}{32t} u\left(\frac{t}{T}\right)\sum_{m_1/d_1=m_2/d_2}\sum_{q\geq1}\tilde{H}(t,d_1q,m_1)\overline{\tilde{H}}(t,d_2q,m_2)dt+O(1)$$
where 
\begin{eqnarray}
& &\tilde{H}(t,d_1q,m_1)\overline{\tilde{H}}(t,d_2q,m_2)\nonumber\\
&=&\sum_{i,j=1}^3\tilde{H}_{1i}(t,d_1q,m_1)\tilde{H}_{2j}(t,d_2q,m_2)\nonumber
\end{eqnarray}
Here, we treat the following one of  the nine terms
\begin{eqnarray}
&&\tilde{H}_{11}(t,d_1q,m_1)\tilde{H}_{21}(t,d_2q,m_2)\nonumber\\
&=&\int_0^1\int_0^1
\frac{1}{\tau\eta(1-\tau)(1-\eta)}\cos(\frac{m_1}{d_1q}t(2\tau-1))\cos(\frac{m_2}{d_2q}t(2\eta-1))\nonumber\\
& &h_1\left(\frac{t\sqrt{\tau(1-\tau)}}{\pi d_1q\sqrt{1+\frac{2\tau
m_1}{d_1q}+\frac{\tau
m_1^2}{d_1^2q^2}}}\right)h_2\left(\frac{t\sqrt{\eta(1-\eta)}}{\pi
d_2q\sqrt{1+\frac{2\eta m_2}{d_2q}+\frac{\eta
m_2^2}{d_2^2q^2}}}\right)d\tau d\eta \nonumber
\end{eqnarray}
For $i=1,2$; 
 $h_i$ are
continuous uniformly on $\mathbb{R}$. For the sum over $q$, we
estimate it as

\begin{eqnarray}
& &\sum_{q\geq1}H_1(t,d_1q,m_1)\overline{H}_2(t,d_2q,m_2)
\nonumber\\
&=&\int_0^1\int_0^1\int_0^{\infty}\cos(\frac{m_1}{d_1q}t(2\tau-1))\cos(\frac{m_2}{d_2q}t(2\eta-1))h_1\left(\frac{t\sqrt{\tau(1-\tau)}}{\pi
d_1q\sqrt{1+\frac{2\tau m_1}{d_1q}+\frac{\tau
m_1^2}{d_1^2q^2}}}\right)\nonumber\\
& &h_2\left(\frac{t\sqrt{\eta(1-\eta)}}{\pi d_2q\sqrt{1+\frac{2\eta
m_2}{d_2q}+\frac{\eta m_2^2}{d_2^2q^2}}}\right)dq
\frac{1}{\tau\eta(1-\tau)(1-\eta)}d\tau d\eta +O(T^{-1}) \nonumber
\end{eqnarray}
\begin{eqnarray}
&=&\int_0^1\int_0^1\int_0^{\infty}\cos(\frac{m_1}{d_1q}t(2\tau-1))\cos(\frac{m_2}{d_2q}t(2\eta-1))h_1\left(\frac{t\sqrt{\tau(1-\tau)}}{\pi
d_1q}\right)\nonumber\\
& &h_2\left(\frac{t\sqrt{\eta(1-\eta)}}{\pi d_2q}\right)dq
\frac{1}{\tau\eta(1-\tau)(1-\eta)}d\tau d\eta +O(T^{-1}) \nonumber\\
&=&\frac{t}{\pi}\int_0^1\int_0^1\int_0^{\infty}\frac{\cos(\frac{\pi
m_1}{d_1}\xi(2\tau-1))\cos(\frac{\pi
m_2}{d_2}\xi(2\eta-1))}{\tau\eta(1-\tau)(1-\eta)}h_1\left(\frac{\xi\sqrt{\tau(1-\tau)}}{
d_1}\right)\nonumber
\end{eqnarray}
\begin{eqnarray}&
&h_2\left(\frac{\xi\sqrt{\eta(1-\eta)}}{d_2}\right)\frac{d\xi}{\xi^2}
d\tau d\eta  +O(T^{-1})\nonumber \end{eqnarray}

Similarly, we can evaluate the other 8 terms and  we obtain the main term of the diagonal term is
$$ T\sum_{\frac{m_1}{d_1}=\frac{m_2}{d_2}}\int_0^{\infty}\int_0^{1}\int_0^{1}
\sum_{i,j=1}^3\tilde{h}_{1i}(\xi,m_1,d_1,\tau_1)\tilde{h}_{2j}(\xi,m_2,d_2,\tau_2)+O(1)\label{d}$$
where 
$$\tilde{h}_{i1}(\xi,m_i,d_i,\tau_i)=\frac{\cos(\frac{\pi m_i}{d_i}\xi(2\tau_i-1)}{\sqrt{\tau_i(1-\tau_i)}})h_i(\frac{\xi\sqrt{\tau_i(1-\tau_i)}}{d_i}),$$
$$\tilde{h}_{i2}(\xi,m_i,d_i,\tau_i)=\frac{\cos(\frac{\pi m_i}{d_i}\xi(2\tau_i-1)}{\tau_i(1-\tau_i)})h_i(\frac{\xi\sqrt{\tau_i(1-\tau_i)}}{d_i})$$
$$\tilde{h}_{i3}(\xi,m_i,d_i,\tau_i)=\frac{\cos(\frac{\pi m_i}{d_i}\xi(2\tau_i-1)}{\tau_i})h_i(\frac{\xi\sqrt{\tau_i(1-\tau_i)}}{d_i})$$
for $i=1, 2$.

\vspace{5mm}
For the non-diagonal terms which is the following
\begin{eqnarray}\sum_{\substack{d_1|m_1\\ d_2|m_2}}\sum_{q_1, q_2}\sum_{c\geq 1}\frac{S(q_1(q_1+\frac{m_1}{d_1}),q_2(q_2+\frac{m_2}{d_2}); c)}{c}\quad\quad\quad\quad\quad\quad\quad\quad\quad\quad\quad\nonumber\\ \quad\quad\quad\quad\quad\quad\quad\quad\times \int_\mathbb{R}J_{2it}\left(\frac{4\pi \sqrt{q_1q_2(q_1+\frac{m_1}{d_1})(q_2+\frac{m_2}{d_2})}}{c}\right)
\frac{\tilde{h}(t)t}{\cosh(\pi t)}dt\nonumber\end{eqnarray}
where$$\tilde{h}(t)=\frac{1}{t^2}H_1(t,d_1q_1,m_1)\overline{H_2(t,d_2q_2,m_2)}u\left(\frac{t}{T}\right),$$
$$H_j(t,k,m)=\int_0^1\left(\frac{1+\frac{m}{k}}{1+\frac{2\tau m}{k}+\frac{\tau m^2}{k^2}}\right)^{it}\frac{1}{\tau(1-\tau)}h_j\left(\frac{t\sqrt{\tau(1-\tau)}}{\pi
k\sqrt{1+\frac{2\tau m}{k}+\frac{\tau m^2}{k^2}}}\right)d\tau;$$
 for $j=1,2.$\\

Let
$x=\frac{4\pi\sqrt{q_1q_2(q_1+\frac{m_1}{d_1})(q_2+\frac{m_2}{d_2})}}{c}$,
the inner integral in the non-diagonal terms is

$$I_T(x)=\frac{1}{2}\int_{\mathbb{R}}\frac{J_{2it}(x)-J_{-2it}(x)}{\sinh(\pi t)}\tilde{h}(t)t\tanh \pi t  dt$$
Since $\tanh (\pi t)=$sgn$ (t)+O(e^{-\pi |t|})$ for large $|t|$ and the function $u$ in $\tilde{h}(t)$  localizes $t$ to $T$, we
can remove $\tanh(\pi t)$ by getting a negligible term $O(T^{-N})$
for any $N>0$. (Note: Here we can truncate the $q_1$, $q_2$, $c$ sums as in the bottom of p.15)

Next we apply the Parseval identity and the Fourier
transform in 
\cite{Ba}
$$\left(\frac{\widehat{J_{2it}(x)-J_{-2it}(x)}}{\sinh(\pi
t)}\right)(y)=-i\cos(x\cosh(\pi y)).$$ By the evaluation of the
Fresnel integrals, we have
\begin{eqnarray}
I_T(x)&=&\frac{-i}{2}\int_0^{\infty}u\left(\frac{t}{T}\right)\sqrt{\frac{2}{xy}}\int_0^1\int_0^1\frac{\cos(\frac{m_1}{d_1k}\sqrt\frac{xy}{2}(2\tau-1))\cos(\frac{m_2}{d_2k}\sqrt\frac{xy}{2}(2\eta-1))}{\tau\eta(1-\tau)(1-\eta)}\nonumber \\
&
&h_1\left(\frac{\sqrt{\frac{xy}{2}}\sqrt{\tau(1-\tau)}}{\pi
d_1k\sqrt{1+\frac{2\tau m_1}{d_1k}+\frac{\tau
m_1^2}{d_1^2k^2}}}\right) h_2\left(\frac{\sqrt{\frac{xy}{2}}\sqrt{\eta(1-\eta)}}{\pi
d_2k\sqrt{1+\frac{2\eta m_2}{d_2k}+\frac{\eta
m_2^2}{d_2^2k^2}}}\right)\nonumber \\
&
&d\tau d\eta\cos
(x-y+\frac{\pi}{4})\frac{dy}{\sqrt{\pi y}}\nonumber
\end{eqnarray}

Thus, the non-diagonal terms are equal to

\begin{eqnarray}
&&\frac{-i}{2}\sum_{\substack{d_1|m_1\\ d_2|m_2}}\sum_{q_1,
q_2}\sum_{c\geq
1}\frac{S(q_1(q_1+\frac{m_1}{d_1}),q_2(q_2+\frac{m_2}{d_2});
c)}{c}\int_0^{\infty}u\left(\frac{\sqrt{\frac{xy}{2}}}{T}\right)\sqrt{\frac{2}{xy}}\int_0^1\int_0^1\nonumber\\
&&\frac{\cos(\frac{m_1}{d_1k}\sqrt\frac{xy}{2}(2\tau-1))\cos(\frac{m_2}{d_2k}\sqrt\frac{xy}{2}(2\eta-1))}{\tau\eta(1-\tau)(1-\eta)}h_1\left(\frac{\sqrt{\frac{xy}{2}}\sqrt{\tau(1-\tau)}}{\pi
d_1q_1}\right)\nonumber\\
&&h_2\left(\frac{\sqrt{\frac{xy}{2}}\sqrt{\eta(1-\eta)}}{\pi
d_2q_2}\right)d\tau d\eta\cos (x-y+\frac{\pi}{4})\frac{dy}{\sqrt{\pi
y}}\nonumber
\end{eqnarray}
Since both $h_1(t)$ and $h_2(t)$ satisfy $h_i^{(n)}\ll(1+|t|)^{-N}$
for any $n>0$ and sufficiently large $N$, and $h_i(t)\ll t^{10}$
when $t\rightarrow 0$, the above sum is concentrated on
$$|\frac{\sqrt{\frac{xy}{2}}}{T}|\ll
1$$
$$T^{-\frac{1}{10}}\ll\frac{xy\tau(1-\tau)}{q_1^2}\ll 1$$
$$T^{-\frac{1}{10}}\ll\frac{xy\eta(1-\eta)}{q_2^2}\ll 1$$

Thus we can get the following
range 
$$\sqrt{\frac{xy}{2}}\sim T.$$
Note that here $x\sim q_1q_2c^{-1}$, the ranges for $q_1$, $q_2$,
$c$ are as follows
$$T\sqrt{\tau(1-\tau)}\ll q_1\ll
T^{\frac{21}{20}}\sqrt{\tau(1-\tau)},$$
$$T\sqrt{\eta(1-\eta)}\ll q_2\ll T^{\frac{21}{20}}\sqrt{\eta(1-\eta)},$$
$$c\ll yT^{\frac{1}{10}}$$

 Here by the above relations and
partial integration sufficiently many times, we will get
sufficiently large power of $y$, $q_1$ and $q_2$ occurring in the
denominator, so we get the terms with $c\gg
T^{\frac{1}{10}}$ contribute $O(1)$.

Denote the above sum  as
$$\sum_{\substack{d_1|m_1\\ d_2|m_2}}\sum_{q_1,
q_2}\sum_{c\geq
1}\frac{S(q_1(q_1+\frac{m_1}{d_1}),q_2(q_2+\frac{m_2}{d_2});
c)}{c}J_{q_1,q_2,c}+O(1).$$

Making the change of variable $t=\frac{\sqrt{\frac{xy}{2}}}{T}$, we
get $J_{q_1,q_2,c}$ is
\begin{eqnarray}
&&\frac{2^{\frac32}}{\sqrt{\pi
x}}\int_0^{\infty} u(t)\frac1t
\sin(-x+\frac{2(tT)^2}{x}-\frac{\pi}{4})\int_0^1\int_0^1\frac{\cos(\frac{m_1}{d_1k}tT(2\tau-1))}{\tau(1-\tau)}\nonumber\\
&&\frac{\cos(\frac{m_2}{d_2k}tT(2\eta-1))}{\eta(1-\eta)}h_1\left(\frac{tT\sqrt{\tau(1-\tau)}}{\pi
d_1q_1}\right)h_2\left(\frac{tT\sqrt{\eta(1-\eta)}}{\pi
d_2q_2}\right)d\tau d\eta dt\nonumber
\end{eqnarray}
By Taylor expansion,
\begin{eqnarray}
xi&=& \frac{4\pi
i}{c}\sqrt{q_1q_2(q_1+\frac{m_1}{d_1})(q_2+\frac{m_2}{d_2})}
\nonumber\\
&=&\frac{2\pi
i}{c}(2q_1q_2+\frac{m_2q_1}{d_2}+\frac{m_1q_2}{d_1}+\cdots)\nonumber
\end{eqnarray}
So we can write 
$$J_{q_1,q_2,c}=\Im (e_c(-(2q_1q_2+\frac{m_2q_1}{d_2}+\frac{m_1q_2}{d_1}))f_c(q_1,q_2)),$$
where
\begin{eqnarray}
f_c(q_1,q_2)&=&
e_c(\frac{m_1m_2}{2d_1d_2}-\frac{m_1^2q_2}{4d_1^2q_1}-\frac{m_2^2q_1}{4d_2^2q_2}+\cdots)\frac{2^{\frac32}}{\sqrt{\pi
x}}\int_0^{\infty} u(t)\frac1t
\nonumber\\
& &e^{i(\frac{2(tT)^2}{x}-\frac{\pi}{4})}\int_0^1\int_0^1
\frac{\cos(\frac{m_1}{d_1k}tT(2\tau-1))\cos(\frac{m_2}{d_2k}tT(2\eta-1))}{\tau\eta(1-\tau)(1-\eta)}\nonumber\\
& &h_1\left(\frac{tT\sqrt{\tau(1-\tau)}}{\pi
d_1q_1}\right)h_2\left(\frac{tT\sqrt{\eta(1-\eta)}}{\pi
d_2q_2}\right)d\tau d\eta dt\nonumber
\end{eqnarray}
and we use the notation $e_c(z)=e^{\frac{2\pi i z}{c}}$.

 Reducing
the summation over $q_1,q_2$ into congruence classes mod $c$, we
have,
\begin{eqnarray}
&
&\sum_{q_1,q_2\geq1}S(q_1(q_1+\frac{m_1}{d_1}),q_2(q_2+\frac{m_2}{d_2});c)e_c(-(2q_1q_2+\frac{m_2q_1}{d_2}+\frac{m_1q_2}{d_1}))f_c(q_1,q_2)\nonumber
\end{eqnarray}

\begin{eqnarray}
&=&\sum_{a,b\mod
c}S(a(a+\frac{m_1}{d_1}),b(b+\frac{m_2}{d_2});c)e_c(-(2ab+\frac{m_2a}{d_2}+\frac{m_1b}{d_1}))\nonumber\\
& &\sum_{q_1\equiv a,q_2\equiv b\mod c}f_c(q_1,q_2)\nonumber\\
&=& \frac{1}{c^2}\sum_{u,v \mod c}\sum_{a,b \mod
c}S(a(a+\frac{m_1}{d_1}),b(b+\frac{m_2}{d_2});c)\nonumber\\
&
&e_c(-(2ab+(\frac{m_2}{d_2}+u)a+(\frac{m_1}{d_1}+v)b))(\sum_{q_1,q_2}f_c(q_1,q_2)e_c(-uq_1-vq_2)).\nonumber
\end{eqnarray}
Apply the Poisson summation for the sum in $q_1$, $q_2$ and obtain,
$$\sum_{q_1,q_2}f_c(q_1,q_2)e_c(-uq_1-vq_2)=\sum_{l_1,l_2}\int\int_{\mathrm{R}^2}f_c(q_1,q_2)e((l_1-\frac{u}{c})q_1+(l_2-\frac{v}{c})q_2)dq_1dq_2.$$
We can assume $|u|\leq\frac{c}{2}$, $|v|\leq\frac{c}{2}$, by partial
integration sufficiently many times, we get
$$\sum_{q_1,q_2}f_c(q_1,q_2)e_c(-uq_1-vq_2)=\int\int_{\mathrm{R}^2}f_c(q_1,q_2)e(-\frac{u}{c}q_1-\frac{v}{c}q_2)dq_1dq_2+O(T^{-A})$$
for any $A>1$.\\
 For $(u, v)\neq(0,0)$, by partial integration sufficiently many
times, we  obtain,  for $c\ll T^{\frac{1}{10}}$
$$\int\int_{\mathrm{R}^2}f_c(q_1,q_2)e(-\frac{u}{c}q_1-\frac{v}{c}q_2)dq_1dq_2\ll T^{-A},$$
for any $A>0$. Thus only $(u, v)=(0,0)$ contributes.  For the $c$-summation,  we can also allow
$c\gg T^{\frac{1}{10}}$, since by
partial integration sufficiently many times,
$$\int\int_{\mathrm{R}^2}f_c(q_1,q_2)dq_1dq_2\ll c^{-A}T^{2},$$
for any $A>0$.\\

For fixed $d_i,m_i$ ($i=1,2$), denote
$$S_c=\sum_{a,b\mod
c}S(a(a+\frac{m_1}{d_1}),b(b+\frac{m_2}{d_2});c)e_c(-(2ab+\frac{m_2a}{d_2}+\frac{m_1b}{d_1}))$$

Thus, the non-diagonal contribution is
\begin{eqnarray}
&&\sum_{\substack{d_1|m_1\\ d_2|m_2}}\sum_{c\geq 1}\Im
(\frac{S_c}{c^2}\int\int_{\mathrm{R}^2}f_c(q_1,q_2)dq_1dq_2)+O(1)\nonumber\\
&=&\sum_{\substack{d_1|m_1\\ d_2|m_2}}\sum_{c\geq 1}\Im
(\frac{S_c}{c^2}\int\int_{\mathrm{R}^2}e_c(\frac{m_1m_2}{2d_1d_2}-\frac{m_1^2q_2}{4d_1^2q_1}-\frac{m_2^2q_1}{4d_2^2q_2})\frac{2^{\frac32}}{\sqrt{\pi
x}}\int_0^{\infty}u(t)\frac1t
\nonumber\\
& &e^{i(\frac{2(tT)^2}{x}-\frac{\pi}{4})}\int_0^1\int_0^1
\frac{\cos(\frac{m_1}{d_1q_1}tT(2\tau-1))\cos(\frac{m_2}{d_2q_2}tT(2\eta-1))}{\tau\eta(1-\tau)(1-\eta)}h_1\left(\frac{tT\sqrt{\tau(1-\tau)}}{\pi
d_1q_1}\right)\nonumber\\
& &h_2\left(\frac{tT\sqrt{\eta(1-\eta)}}{\pi d_2q_2}\right)d\tau
d\eta dtdq_1dq_2)+O(1)\nonumber\\
&=&T\sum_{\substack{d_1|m_1\\
d_2|m_2}}\sum_{c\geq 1}\Im
(\frac{S_c\zeta_8}{c^{\frac{3}{2}}}\int\int_{\mathrm{R}^2}e_c(\frac{m_1m_2}{2d_1d_2}-\frac{m_1^2\phi}{4d_1^2\xi}-\frac{m_2^2\xi}{4d_2^2\phi})\frac{2^{\frac32}}{(\xi\phi)^{\frac32}}\nonumber\\
& &e(\xi\phi c) \int_0^1\int_0^1
\frac{\cos(\frac{m_1\xi}{d_1}(2\tau-1))\cos(\frac{m_2\phi}{d_2}(2\eta-1))}{\tau\eta(1-\tau)(1-\eta)}h_1\left(\frac{\xi\sqrt{\tau(1-\tau)}}{\pi
d_1}\right)\nonumber\\
& &h_2\left(\frac{\phi\sqrt{\eta(1-\eta)}}{\pi d_2}\right)d\tau
d\eta d\xi d\phi)+O(1).\nonumber
\end{eqnarray}
Note:   For the last coefficient $T$ comes from another change of variable. The contribution from the higher Taylor coefficients in the definition of $f_c(q_1, q_2)$ are of order roughly $O(1/T)$, hence negligible by partial integration sufficiently many times.  

Thus, we obtain the following asymptotic formula including the diagonal and non-diagonal terms:
\begin{eqnarray}
&&\lim_{T\rightarrow\infty}\frac1T\sum_{j\ge 1}u\left(\frac{t_j}{T}\right)L (1,\mathrm{sym}^2\varphi_j)\omega_j(P_{h_1,m_1,2})\overline{\omega}_j(P_{h_2,m_2,2})\nonumber\\
&=&
\int_0^{\infty}u(t)dt(\sum_{\frac{m_1}{d_1}=\frac{m_2}{d_2}}\int_0^{\infty}\int_0^{1}\int_0^{1}
\sum_{i,j=1}^3\tilde{h}_{1i}(\xi,m_1,d_1,\tau_1)\tilde{h}_{2j}(\xi,m_2,d_2,\tau_2)d\tau_1 d\tau_2 \frac{d\xi}{\xi^2}\label{d}\nonumber\\
&&\\
&&+\sum_{d_1|m_1,d_2|m_2}\sum_{c\geq 1}\int_0^{\infty}\int_0^{\infty}\int_0^{1}\int_0^{1}\Im
\{\frac{S_c\zeta_8}{c^{\frac{5}{2}}}e_c(\frac{m_1m_2}{2d_1d_2}-\frac{m_1^2\xi_1}{4d_1^2\xi_2}-\frac{m_2^2\xi_2}{4d_2^2\xi_1})\nonumber
\end{eqnarray}
\begin{eqnarray}
&&e((d_1d_2)^2\xi_1\xi_2 c)\}\sum_{i,j=1}^3\tilde{h}_{1i}(\xi_1,m_1,d_1,\tau_1)\tilde{h}_{2j}(\xi_2,m_2,d_2,\tau_2)d\tau_1d\tau_2\frac{d\xi_1d\xi_2}{(\xi_1\xi_2)^{3/2}})
\nonumber\label{non}\\
&&\yesnumber
\end{eqnarray}
where 
$$\tilde{h}_{i1}(\xi,m_i,d_i,\tau_i)=\frac{\cos(\frac{\pi m_i}{d_i}\xi(2\tau_i-1)}{\sqrt{\tau_i(1-\tau_i)}})h_i(\frac{\xi\sqrt{\tau_i(1-\tau_i)}}{d_i}),$$
$$\tilde{h}_{i2}(\xi,m_i,d_i,\tau_i)=\frac{\cos(\frac{\pi m_i}{d_i}\xi(2\tau_i-1)}{\tau_i(1-\tau_i)})h_i(\frac{\xi\sqrt{\tau_i(1-\tau_i)}}{d_i})$$
$$\tilde{h}_{i3}(\xi,m_i,d_i,\tau_i)=\frac{\cos(\frac{\pi m_i}{d_i}\xi(2\tau_i-1)}{\tau_i}h_i(\frac{\xi\sqrt{\tau_i(1-\tau_i)}}{d_i})$$
for $i=1, 2$.
 
 In the non-diagonal terms (\ref{non}), $S_c$ is a sum involving Kloosterman sums which is explicitly
 $$S_c=\sum_{a,b \mod
c}S(a(a+\frac{m_1}{d_1}),b(b+\frac{m_2}{d_2});c)e_c(-(2ab+\frac{m_2a}{d_2}+\frac{m_1b}{d_1}))$$

This gives the existence of the limiting variance for the case $k_1=k_2=1$.

Now, by the induction and the recurrence formula 
$$A_{k+1}(s)=-2kA_{k}(s)+2A_{k}(s+1)-[(k-\frac 12)^2+t_j^2]A_{k-1}(s)$$
we can obtain the existence of $B(P_{h_1,m_1,k_1}, P_{h_2,m_2,k_2})$ for any $k_1, k_2\in \mathbb{Z}$. 
Precisely, for the term $[(k-\frac 12)^2+t_j^2]A_{k-1}(s)$, the involving Gamma factors are,
 \begin{eqnarray}
&&\frac{[(k-\frac 12)^2+t_j^2]\Gamma(\frac12+it_j)A_{k-1}(s)}{\Gamma(k+\frac32+it_j)}\nonumber\\
&=&\frac{\Gamma(\frac12+it_j)A_{k-1}(s)}{\Gamma(k-\frac12+it_j)}\cdot\frac{[(k-\frac 12)^2+t_j^2]\Gamma(k-\frac12+it_j)}{\Gamma(k+\frac32+it_j)}\nonumber
 \end{eqnarray}
Thus, we can evaluate using the induction assumption for the first factor and Stirling formula for the second factor.

For the term $kA_{k}(s)$, we can use the similar argument to evaluate. While for the terms involving $A_0(s+k)$ and $B(s+k)$, the Gamma factors are easy to handle since they are simply 
$$\frac{\Gamma(\frac{s+k}{2})^2}{\Gamma(s+k)}, \quad \frac{\Gamma(\frac{s+k}{2})\Gamma(\frac{s+k}{2}+1)}{\Gamma(s+k+1)}$$  

Moreover, by keeping track of the dependence on $h_1$ and $h_2$ and integration by parts in the  double integrals of (\ref{d}) and (\ref{non}), we obtain that there is a constant $A$ (depending on $k_1$, $k_2$), such that the sesquilinear form Q satisfies
\begin{eqnarray}
&&|Q(P_{h_1,m_1,k_1},P_{h_2,m_2,k_2})|\ll_{k_1,k_2} ((|m_1|+1)(|m_2|+1))^A\|h_1\|_A\|h_2\|_A.\nonumber\\
&&\yesnumber
\end{eqnarray}

If any incomplete Poincar\'{e} series in this proposition is
replaced by incomplete Eisenstein  series, i.e. $m_i=0$ with 
mean zero satisfying (9), the proposition is still valid. For the case $m_1=m_2=0$, there is
a slight modification for $Q$ as follows.
\begin{eqnarray}
&&\sum_{d_1,d_2\geq 1}\int_0^{\infty}\int_0^1
h_1\left(\frac{\xi\sqrt{\tau(1-\tau)}}{
d_1}\right)d\tau \int_0^1
h_2\left(\frac{\xi\sqrt{\eta(1-\eta)}}{
d_2}\right)d\eta \frac{d\xi}{\xi^2}\nonumber\\
&=&\int_0^{\infty}\int_0^1 \sum_{d_1\geq
1}h_1\left(\frac{\xi\sqrt{\tau(1-\tau)}}{
d_1}\right)d\tau \int_0^1 \sum_{d_2\geq
1}h_2\left(\frac{\xi\sqrt{\eta(1-\eta)}}{
d_2}\right)d\eta \frac{d\xi}{\xi^2}\nonumber
\end{eqnarray}
By Euler-MacLaurin summation formula, we have

$$\sum_{d_1\geq
1}h_1\left(\frac{\xi\sqrt{\tau(1-\tau)}}{
d_1}\right)=-\int_0^{\infty}b_2(\alpha)H_1\left(\frac{\xi\sqrt{\tau(1-\tau)}}{
\alpha}\right)\frac{d\alpha}{\alpha^2},$$ where $b_2(\alpha)$ is the
Bernoulli polynomial of degree 2, $H_1(x)=(h_1'(x)x^2)'$. For the
sum over $d_2$, we have the similar expression.

\end{proof}

This completes the proof of the existence of the quantum variance for vectors $\psi_1=P_{h_1,m_1,2k_1}$ and $\psi_2=P_{h_2,m_2,2k_2}$ in Theorem 1. To obtain the result for the general $\psi_1$, $\psi_2$ asserted in the Theorem one proceeds by the approximation arguments in Section 4 of \cite{Luo2}, which requires keeping track of the dependence of the remainders in the analysis leading to (33) and (34) above. This is a straightforward generalization and we omit the details. In the next section we derive an explicit version of (33) and (34)   for special Poincar\'{e} series of various weights.

\section{Symmetry Properties of $Q$} 

We begin by showing that the sesquilinear form $Q$ is invariant under the geodesic flow as well as  under time reversal. This is true much more generally as can be seen from the recent work of Anatharaman and Zelditch \cite{AZ} in the context of $\Gamma\backslash\mathbb{H}$ where $\Gamma$ is any lattice (not just $SL_2(\mathbb{Z})$, in fact they deal with cocompact lattices but their results are easily extended to finite volume as in \cite{Zelditch}). In this generality, they relate the Wigner distributions to what they call Patterson-Sullivan distributions. Since the latter are geodesic flow as well time reversal invariant, this yields a complete asymptotic expansion measuring this invariance.  This is given in their Theorem 1.2 and the expansion on page 386 (note
that our quantization and those in \cite{AZ} and \cite{AZ2} all coincide). 
Taken to second order this reads:

If $f$ is smooth on $\Gamma\backslash G$ as in Theorem 1, i.e.  bounded and with rapidly decay at cusps, let $\tau\in\mathbb{R}$ are fixed and $f_{\tau}(x)=f(x\mathcal{G}_{\tau})$, where  $\mathcal{G}_{\tau}$ is the geodesic flow, then
\begin{eqnarray}
&&<Op(f_{\tau})\phi_j,\phi_j>\nonumber\\
&=&<Op(f)\phi_j,\phi_j>+\frac{<Op(L_2(f_{\tau}-f))\phi_j,\phi_j>}{t_j}+O(\frac{1}{t_j^2})
\end{eqnarray}
where $L_2$ is a second order differential operator generated by the vector field  \(\displaystyle X_+=\left(\begin{array}{cc} 0
& 1\\
0& 0\end{array}\right) \).  Note: here we interchangeably use the notations of $<Op(f)\phi_j,\phi_j>$ and $\omega_j(f).$

First we apply (36) with the first term only, that is 
\begin{eqnarray}
<Op(f_{\tau})\phi_j,\phi_j>
=<Op(f)\phi_j,\phi_j>+O(\frac{1}{t_j})
\end{eqnarray}
to the variance sums.
\begin{eqnarray}
&&\sum_{t_j\le T}<Op(f_{\tau})\phi_j,\phi_j>\overline{<Op(g)\phi_j,\phi_j>}\nonumber\\
&=&\sum_{t_j\le T}<Op(f)\phi_j,\phi_j>\overline{<Op(g)\phi_j,\phi_j>}+O(\sum_{t_j\le T}\frac{1}{t_j}|<Op(g)\phi_j,\phi_j>|)\nonumber\\
\end{eqnarray}
Now the general quantum ergodicity theorem in this context \cite{Zel} asserts that as $y\rightarrow\infty$,
\begin{eqnarray}
\sum_{t_j\le y}|<Op(g)\phi_j,\phi_j>|=o(y^2)
\end{eqnarray}
Hence by partial summation in the second sum in (38), we get that  
\begin{eqnarray}
&&\sum_{t_j\le T}<Op(f_{\tau})\phi_j,\phi_j>\overline{<Op(g)\phi_j,\phi_j>}\nonumber\\
&=&\sum_{t_j\le T}<Op(f)\phi_j,\phi_j>\overline{<Op(g)\phi_j,\phi_j>}+o(T)
\end{eqnarray}
A similar statement is true if $f_{\tau}$ is replaced by time reversal applied to $f$. Hence in this generality (and with no arithmetic assumptions) the quantum variance sums are geodesic flow and time reversal invariant to the order required in our Theorem 1, in which the quantum sum has an error term $o(1).$ 

In our arithmetic setting of $\Gamma=SL_2(\mathbb{Z})$ we can use Theorem 1 together with the relation (36) (to second order) to deduce (with or without the arithmetic weights) that as $T\rightarrow\infty$,
\begin{eqnarray}
&&\sum_{t_j\le T}<Op(f_{\tau})\phi_j,\phi_j>\overline{<Op(g)\phi_j,\phi_j>}-\sum_{t_j\le T}<Op(f)\phi_j,\phi_j>\overline{<Op(g)\phi_j,\phi_j>}\nonumber\\
&=&Q(Op(L_2(f_{\tau}-f)), g)\log T+o(\log T)\nonumber
\end{eqnarray}
In any case we deduce from  the above that $Q$ is bilinearly invariant under both the geodesic flow and time reversal.

Therefore, from the symmetry consideration as in Luo-Rudnick-Sarnak \cite{Luo0}, we know that the space of such Hermitian forms $Q(f,g)$ restricted to subspaces associated to each  representation $U_{\pi_j^k}$ is at most one dimensional. 

To use this further, we need show the orthogonality that $Q(\phi_j,\phi_k)=0$ if  $\phi_j$, $\phi_k$ are  in the different irreducible representations $\pi_j$, $\pi_k$.  It suffices to show for the generator vectors of  the representation, i.e. $Q(\phi_j,\phi_k)=0$ if $\phi_j,\phi_k$ is either 
	 holomorphic form or Maass form. To show this, we need first evaluate $Q(\phi_j,\phi_k)$ and then use the explicit Hermitian form $Q$ to deduce the self-adjointness with respect to Hecke operators. We consider the following three cases:

(a)  Both $\phi_j$ and $\phi_k$ are holomorphic;

(b) $\phi_j$ is holomorphic and $\phi_k$  is Maass form;

(c) Both $\phi_j$ and $\phi_k$ are Maass forms, while this case was dealt in \cite{Z}.

In case (a), we first use holomorphic Poincar\'{e} series to find an explicit form of $Q(P_{m_1,k_1},P_{m_2,k_2})$. 


For holomorphic Poincar\'{e} series
$$P_{m,k}(z)=\sum_{\gamma\in\Gamma_\infty\backslash\Gamma}j(\gamma, z)^{-k}e(m(\gamma
z)).$$
By unfolding, we have
\begin{eqnarray} <P_{m,k}, d\omega_j> 
&=& \int_{\Gamma_{\infty}\backslash
\mathbb{H}}e^{-2\pi my}e(mx)\varphi_j(z)\varphi_{j,k}(z)d\mu(z)
  \end{eqnarray}
Apply the Fourier expansion of  $\varphi_{j,k}(z)$ \cite{J94},
$$
\varphi_{j,k}(z)=(-1)^k\Gamma(1/2+it_j)\sum_{n\neq 0}\frac{c_j(|n|)W_{\sgn(n)k,it_j}(4\pi|n|y)e(nx)}{\sqrt{|n|}\Gamma(\frac12+\sgn(n)k+it_j)},
$$
and $$\varphi_j(z)=\sum_{n\neq 0}\frac{c_j(|n|)}{\sqrt{|n|}}W_{0,it_j}(4\pi|n|y)e(nx).$$
From the relation $c_j(n)=c_j(1)\lambda_j(n)$ and the well-known multiplicativity of  Hecke eigenvalues
$$\lambda_j(n)\lambda_j(m)=\sum_{d|(n,m)}\lambda_j\left(\frac{mn}{d^2}\right),$$
we have
\begin{eqnarray}
 <P_{m,k}, d\omega_j> &=&4\pi(-1)^k\Gamma(\frac12+it_j)c_j(1)\sum_{d|m}\sum_{q\neq0,-\frac md}\frac{c_j(q^2+\frac{qm}{d})}{\sqrt{|1+\frac{m}{qd}|}}\nonumber\\
 &&\int_0^{\infty}\frac{W_{\sgn(q)k,it_j}(y)}{\Gamma(\frac12+\sgn(q)k+it_j)}W_{0,it_j}\left(y(1+\frac{m}{qd})\right)\left(\frac{y}{qd}\right)^{k}e^{\left(\frac{-my}{2qd}\right)}\frac{dy}{y^2}.\label{4}
\end{eqnarray}
 
For the inner integral, we apply  the formula 7.671 in \cite{gr} 
\begin{eqnarray}
&&\int_0^{\infty}x^{-k-\frac32}e^{-\frac12(a-1)x}K_{\mu}(\frac12 ax)W_{k,\mu}(x)dx\nonumber\\
&=&\frac{\pi\Gamma(-k)\Gamma(2\mu-k)\Gamma(-2\mu-k)}{\Gamma(\frac12-k)\Gamma(\frac12+\mu-k)\Gamma(\frac12-\mu-k)}2^{2k+1}a^{k-\mu}F(-k,2\mu-k;-2k;1-\frac1a)\nonumber
\end{eqnarray}
by letting $a=1+m/d$, $\mu=it_j$ and for the hypergeometric series $F(-k,2\mu-k;-2k;1-\frac1a)$, we use 9.111 in \cite{gr}
\begin{eqnarray}
F(\alpha,\beta;\gamma;z)=\frac{1}{B(\beta,\gamma-\beta)}\int_0^{1}t^{\beta-1}(1-t)^{\gamma-\beta-1}(1-tz)^{-\alpha}dt\nonumber
\end{eqnarray}
By Stirling formula and similar method of calculating $<P_{h,m,k}, d\omega_j> $ in Section 2, we have 
\begin{eqnarray}
&& <P_{m,k}, d\omega_j> \nonumber\\
&=&\frac{1}{L(1,\textrm{sym}^2\varphi_j)}\sum_{d|m}\sum_{q>0}\lambda_j(q^2+\frac{qm}{d})\int_0^1\left(\frac{(1+\frac{m}{qd})}{1+\frac{2\tau m}{qd}+\tau(\frac{m}{qd})^2}\right)^{it_j}\nonumber\\
 &&(\tau(1-\tau)(1+\frac{2\tau m}{qd}+\tau(\frac{m}{qd})^2))^{k-\frac 12} \exp\left(\frac{-mt_j\sqrt{\tau(1-\tau)}}{2 dq\sqrt{1+\frac{2\tau m}{qd}+\frac{\tau m^2}{(qd)^2}}}\right)d\tau \nonumber\end{eqnarray}
By the similar treatment on Kuznetsov formula as we did in \cite{Z}, 
we obtain

\begin{IEEEeqnarray*}{rCl}
&&\lim_{T\rightarrow\infty}\frac1T\sum_{j\ge 1}u\left(\frac{t_j}{T}\right)L (1,\mathrm{sym}^2\varphi_j)\omega_j(P_{m_1,k_1})\overline{\omega}_j(P_{m_2,k_2})\nonumber\\
&=&
\int_0^{\infty}u(t)dt\sum_{\frac{m_1}{d_1}=\frac{m_2}{d_2}}\int_0^{\infty}\int_0^1
\cos(\frac{\pi
m_1}{d_1}\xi(2\tau-1))\exp\left(\frac{-m_1\xi\sqrt{\tau(1-\tau)}}{
d_1}\right) \nonumber\\
&&(\tau(1-\tau))^{k_1}d\tau\int_0^1\cos(\frac{\pi m_2}{d_2}\xi(2\eta-1))
\exp\left(\frac{-m_2\xi\sqrt{\eta(1-\eta)}}{
d_2}\right)(\eta(1-\eta))^{k_2}\\&&\cdot d\eta
\xi^{k_1+k_2}\frac{d\xi}{\xi^2}+\int_0^{\infty}u(t)dt\sum_{\substack{d_1|m_1\\
d_2|m_2}}\sum_{c\geq 1} \int\int_{\mathrm{R}^2}\Im
(\frac{S_c\zeta_8}{c^{\frac{3}{2}}}e_c(\frac{m_1m_2}{2d_1d_2}-\frac{m_1^2\xi}{4d_1^2\phi}-\frac{m_2^2\phi}{4d_2^2\xi})\nonumber\\
& &\frac{2^{\frac32}\xi^{k_1}\phi^{k_2}}{(\xi\phi)^{\frac32}}e((d_1d_2)^2\xi\phi c)) \int_0^1\int_0^1 \cos(\pi
m_1d_2\xi(2\tau-1))\cos(\pi
m_2d_1\phi(2\eta-1))\nonumber\\
& &\tau^{k_1}\eta^{k_2}(1-\tau)^{k_1}(1-\eta)^{k_2}\exp(-m_1\xi d_2\sqrt{\tau(1-\tau)})\exp(-m_2\phi
d_1\sqrt{\eta(1-\eta)})\nonumber\\
&&d\tau
d\eta d\xi d\phi
\end{IEEEeqnarray*}

Now, we can use this explicit form to show the self-adjointness of $B(\phi_j,\phi_k)$ with respect to  Hecke operators  for holomorphic $\phi_j,$ $\phi_k$, 
in fact we can check it for each Hecke operator
$T_p$, where $p$ is a prime, i.e.
\begin{proposition}
$$Q(T_pP_{m_1,k_1},P_{m_2,k_2})=Q(P_{m_1,k_1},T_pP_{m_2,k_2}).$$
\end{proposition}

\begin{proof} This is a direct generalization of Appendix A.3 in \cite{Luo2}, which deals with 
the Maass case with $k=0$. We use the fact (Theorem 6.9 in \cite{iwa})
\begin{eqnarray}  T_nP_{m,k}(z)=\sum_{d|(m,n)}\left(\frac{n}{d}\right)^{k-1}P_{\frac{mn}{d^2},k}(z),\label{14}\end{eqnarray}
and the explicit evaluation of
$S_{c,\frac{m_1}{d_1},\frac{m_2}{d_2}}(\gamma)$ (Appendix A.2 in \cite{Luo2})to verify it.

We denote $$Q(P_{m_1,k_1},P_{m_2,k_2})=Q_D(P_{m_1,k_1},P_{m_2,k_2})+Q_{ND}(P_{m_1,k_1},P_{m_2,k_2})$$ as the diagonal and non-diagonal terms, and we consider the following 4 cases:

(i) If $p\nmid m_1m_2$,
$Q_{D}(T_pP_{m_1,k_1},P_{m_2,k_2})=Q_{D}(P_{m_1,k_1},T_pP_{m_2,k_2})$;

(ii) If $p\nmid m_1m_2$,
$Q_{ND}(T_pP_{m_1,k_1},P_{m_2,k_2})=Q_{ND}(P_{m_1,k_1},T_pP_{m_2,k_2})$;

(iii) If $p^a\parallel(m_1,m_2)$,
$Q_{D}(T_pP_{m_1,k_1},P_{m_2,k_2})=Q_{D}(P_{m_1,k_1},T_pP_{m_2,k_2})$;

(iv) If $p^a\parallel(m_1,m_2)$,
$Q_{ND}(T_pP_{m_1,k_1},P_{m_2,k_2})=Q_{ND}(P_{m_1,k_1},T_pP_{m_2,k_2})$.

To prove (i), we use the fact
$$T_pP_{m,k}(z)=p^{k-1}P_{pm,k}(z)$$ from (\ref{14}). Also, from the conditions $d_1|pm_1,$ $d_2|m_2$ and
$\frac{pm_1}{d_1}=\frac{m_2}{d_2}$ we have $p|d_1$. For our convenience, we denote $$\tilde{h}(\frac{m_i\xi}{d_i},k_i,\tau_i)=\cos(\frac{\pi
m_i}{d_i}\xi(2\tau_i-1))\exp\left(\frac{-m_i\xi\sqrt{\tau_i(1-\tau_i)}}{
d_i}\right)(\tau_i(1-\tau_i))^{k_i} $$Thus, by making the
change of variables $d_1\rightarrow pd_1$, $\frac{\xi}{p}\rightarrow \xi$ and $d_2\rightarrow pd_2$, $\frac{\xi}{p}\rightarrow \xi$ for $Q_{D}(P_{pm_1,k},P_{m_2,k})$ and $Q_{D}(P_{m_1,k},P_{pm_2,k})$ respectively, 
we have
\begin{eqnarray}
&
&Q_{D}(T_pP_{m_1,k_1},P_{m_2,k_2})\nonumber\\
&=&p^{k_1-1}Q_{D}(P_{pm_1,k_1},P_{m_2,k_2})\nonumber\\
&=&p^{-1}\sum_{\frac{m_1}{d_1}=\frac{m_2}{d_2}}\int_0^{\infty}\int_0^{1}\int_0^{1}\prod_{i=1}^2\tilde{h}(\frac{m_i\xi}{d_i},l_i,\tau_i)d\tau_i\frac{\xi^{k_1+k_2}d\xi}{\xi^2}\nonumber\\
&=&p^{k_2-1}Q_{D}(P_{m_1,k_1},P_{pm_2,k_2})\nonumber\\&=&Q_{D}(P_{m_1,k_1},T_pP_{m_2,k_2}).\nonumber
\end{eqnarray}

For (ii), we have
\begin{eqnarray}
&
&Q_{ND}(T_pP_{m_1,k_1},P_{m_2,k_2})\nonumber\\
&=&p^{k_1-1}Q_{ND}(P_{pm_1,k_1},P_{m_2,k_2})\nonumber\\
&=&p^{k_1-1}\sum_{l_1=0}^{k_1}\sum_{l_2=0}^{k_2}\sum_{\substack{d_1|pm_1\\
d_2|m_2}}\sum_{c\geq 1}\int_0^{\infty}\int_0^{\infty}\int_0^{1}\int_0^{1}\Im
\{\frac{S_c\zeta_8}{c^{\frac{5}{2}}}e_c(\frac{pm_1m_2}{2d_1d_2}-\frac{p^2m_1^2\xi}{4d_1^2\phi}-\frac{m_2^2\phi}{4d_2^2\xi})\nonumber\\
&&e((d_1d_2)^2\xi\phi c)\}\tilde{h}(\frac{p\xi_1 m_1}{d_1},l_1,\tau_1)\tilde{h}(\frac{\xi_2 m_2}{d_2},l_2,\tau_2)d\tau_1d\tau_2\frac{d\xi_1d\xi_2}{\xi_1^{3/2-k_1}\xi_2^{3/2-k_2}}\nonumber
\end{eqnarray}
\begin{eqnarray}
&=&p^{-1}\sum_{l_1=0}^{k_1}\sum_{l_2=0}^{k_2}\sum_{\substack{d_1|m_1\\
d_2|m_2}}\sum_{c\geq 1}\int_0^{\infty}\int_0^{\infty}\int_0^{1}\int_0^{1}\Im
\{\frac{\tilde{S}_c\zeta_8}{c^{\frac{5}{2}}}e_c(\frac{pm_1m_2}{2d_1d_2}-\frac{p^2m_1^2\xi}{4d_1^2\phi}-\frac{m_2^2\phi}{4d_2^2\xi})\nonumber\\
&&e((d_1d_2)^2\xi\phi c)\}\tilde{h}(\frac{p\xi_1 m_1}{d_1},l_1,\tau_1)\tilde{h}(\frac{\xi_2 m_2}{d_2},l_2,\tau_2)d\tau_1d\tau_2\frac{d\xi_1d\xi_2}{\xi_1^{3/2-k_1}\xi_2^{3/2-k_2}}\nonumber\\
&&+p^{-1}\sum_{l_1=0}^{k_1}\sum_{l_2=0}^{k_2}\sum_{\substack{d_1|m_1\\
d_2|m_2}}\sum_{c\geq 1}\int_0^{\infty}\int_0^{\infty}\int_0^{1}\int_0^{1}\Im
\{\frac{S_c\zeta_8}{c^{\frac{5}{2}}}e_c(\frac{m_1m_2}{2d_1d_2}-\frac{m_1^2\xi}{4d_1^2\phi}-\frac{m_2^2\phi}{4d_2^2\xi})\nonumber\\
&&e((d_1d_2)^2\xi\phi c)\}\tilde{h}(\frac{\xi_1 m_1}{d_1},l_1,\tau_1)\tilde{h}(\frac{\xi_2 m_2}{d_2},l_2,\tau_2)d\tau_1d\tau_2\frac{d\xi_1d\xi_2}{\xi_1^{3/2-k_1}\xi_2^{3/2-k_2}}\nonumber
\end{eqnarray}
The above two sums correspond to the conditions $p\nmid d_1,$ and
$p|d_1$ respectively.

Similarly, we have
\begin{eqnarray}
&
&Q_{ND}(P_{m_1,k_1},T_pP_{m_2,k_2})\nonumber\\
&=&p^{-1}\sum_{l_1=0}^{k_1}\sum_{l_2=0}^{k_2}\sum_{\substack{d_1|m_1\\
d_2|m_2}}\sum_{c\geq 1}\int_0^{\infty}\int_0^{\infty}\int_0^{1}\int_0^{1}\Im
\{\frac{\tilde{S}'_c\zeta_8}{c^{\frac{5}{2}}}e_c(\frac{pm_1m_2}{2d_1d_2}-\frac{m_1^2\xi}{4d_1^2\phi}-\frac{p^2m_2^2\phi}{4d_2^2\xi})\nonumber\\
&&e((d_1d_2)^2\xi\phi c)\}\tilde{h}(\frac{\xi_1 m_1}{d_1},l_1,\tau_1)\tilde{h}(\frac{p\xi_2 m_2}{d_2},l_2,\tau_2)d\tau_1 d\tau_2\frac{d\xi_1d\xi_2}{\xi_1^{3/2-k_1}\xi_2^{3/2-k_2}}\nonumber\\
&&+p^{-1}\sum_{l_1=0}^{k_1}\sum_{l_2=0}^{k_2}\sum_{\substack{d_1|m_1\\
d_2|m_2}}\sum_{c\geq 1}\int_0^{\infty}\int_0^{\infty}\int_0^{1}\int_0^{1}\Im
\{\frac{S_c\zeta_8}{c^{\frac{5}{2}}}e_c(\frac{m_1m_2}{2d_1d_2}-\frac{m_1^2\xi}{4d_1^2\phi}-\frac{m_2^2\phi}{4d_2^2\xi})\nonumber\\
&&e((d_1d_2)^2\xi\phi c)\}\tilde{h}(\frac{\xi_1 m_1}{d_1},l_1,\tau_1)\tilde{h}(\frac{\xi_2 m_2}{d_2},l_2,\tau_2)d\tau_1d\tau_2\frac{d\xi_1d\xi_2}{\xi_1^{3/2-k_1}\xi_2^{3/2-k_2}}\nonumber\end{eqnarray}
 Make the change of
variables $\xi\rightarrow\frac{\xi}{p}$, $\phi\rightarrow p\phi$. Moreover, by the evaluation of the sum $S_c$ which involving the Salie sum, precisely $$S_{c,pm_1/d_1,m_2/d_2}=S_{c,m_1/d_1,pm_2/d_2}.$$ We
can see
$Q_{ND}(T_pP_{m_1,k_1},P_{m_2,k_2})=Q_{ND}(P_{m_1,k_1},T_pP_{m_2,k_2})$.

For the cases (iii) and (iv), we use the fact
$$T_pP_{m,k}(z)=p^{k-1}P_{pm,k}(z)+P_{\frac{m}{p},k}(z).$$
where if $p\nmid m$, we understand that
$P_{h(\frac{\cdot}{p}),\frac{m}{p}}(z)=0$.

Thus, for the case (iii), we have
\begin{eqnarray}
&&Q_{\infty}(T_pP_{h_1,m_1,k_1},P_{h_2,m_2,k_2})\nonumber\\
&=&p^{k_1-1}Q_{\infty}(P_{h_1(p\cdot),pm_1,k_1},P_{h_2,m_2,k_2})+Q_{\infty}(P_{h_1(\frac{\cdot}{p}),\frac{m_1}{p},k_1},P_{h_2,m_2,k_2})\nonumber\\
&=&A+B\nonumber
\end{eqnarray}
Similarly,
\begin{eqnarray}
&&Q_{D}(P_{m_1,k_1},T_pP_{m_2,k_2})\nonumber\\
&=&p^{k_2-1}Q_{D}(P_{m_1,k_1},P_{pm_2,k_2})+Q_{D}(P_{\frac{m_1}{p},k_1},P_{\frac{m_2}{p},k_2})\nonumber\\
&=&A_1+Q_1\nonumber
\end{eqnarray}

We can check that $$A(p|d_1)=A_1(p|d_2),$$
$$A(p\nmid d_1)=B_1(p\nmid d_1),$$
$$B(p\nmid d_2)=A_1(p\nmid d_2),$$
$$B(p|d_2)=B_1(p|d_1).$$
Hence, we get (iii).

 The proof of (iv) is the most tedious one and we will use
the induction to prove that. We have
\begin{eqnarray}
&
&Q_{ND}(T_pP_{m_1,k_1},P_{m_2,k_2})\nonumber\\
&=&p^{k_1-1}Q_{ND}(P_{pm_1,k_1},P_{m_2,k_2})+Q_{ND}(P_{\frac{m_1}{p},k_1},P_{m_2,k_2})\nonumber
\end{eqnarray}
From the expression of $Q(P_1,P_2)$, it equals
\begin{eqnarray}
&&p^{k_1-1}\sum_{l_1=0}^{k_1}\sum_{l_2=0}^{k_2}\sum_{\substack{d_1|pm_1\\
d_2|m_2}}\sum_{c\geq 1}\int_0^{\infty}\int_0^{\infty}\int_0^{1}\int_0^{1}\Im
\{\frac{S_c\zeta_8}{c^{\frac{5}{2}}}e_c(\frac{pm_1m_2}{2d_1d_2}-\frac{p^2m_1^2\xi}{4d_1^2\phi}-\frac{m_2^2\phi}{4d_2^2\xi})\nonumber\\
&&e((d_1d_2)^2\xi\phi c)\}\tilde{h}(\frac{p\xi_1 m_1}{d_1},l_1,\tau_1)\tilde{h}(\frac{\xi_2 m_2}{d_2},l_2,\tau_2)d\tau_1 d\tau_2\frac{d\xi_1d\xi_2}{\xi_1^{3/2-k_1}\xi_2^{3/2-k_2}}\nonumber\\
&&+\sum_{l_1=0}^{k_1}\sum_{l_2=0}^{k_2}\sum_{\substack{d_1|m_1/p\\
d_2|m_2}}\sum_{c\geq 1}\int_0^{\infty}\int_0^{\infty}\int_0^{1}\int_0^{1}\Im
\{\frac{S_c\zeta_8}{c^{\frac{5}{2}}}e_c(\frac{m_1m_2}{2pd_1d_2}-\frac{m_1^2\xi}{4p^2d_1^2\phi}-\frac{m_2^2\phi}{4d_2^2\xi})\nonumber\\
&&e((d_1d_2)^2\xi\phi c)\}\tilde{h}(\frac{m_1\xi_1/p}{d_1},l_1,\tau_1)\tilde{h}(\frac{\xi_2m_2}{d_2},l_2,\tau_2)d\tau_1d\tau_2\frac{d\xi_1d\xi_2}{\xi_1^{3/2-k_1}\xi_2^{3/2-k_2}}\nonumber
\end{eqnarray}

We denote the above sum as $I_1+I_2$. Similarly,
\begin{eqnarray}
&
&Q_{ND}(P_{m_1,k_1},T_pP_{m_2,k_2})\nonumber\\
&=&p^{k_2-1}Q_{ND}(P_{m_1},P_{pm_2})+Q_{ND}(P_{m_1},P_{\frac{m_2}{p}})\nonumber\\
&=&p^{k_2-1}\sum_{l_1=0}^{k_1}\sum_{l_2=0}^{k_2}\sum_{\substack{d_1|m_1\\
d_2|pm_2}}\sum_{c\geq 1}\int_0^{\infty}\int_0^{\infty}\int_0^{1}\int_0^{1}\Im
\{\frac{S_c\zeta_8}{c^{\frac{5}{2}}}e_c(\frac{pm_1m_2}{2d_1d_2}-\frac{m_1^2\xi}{4d_1^2\phi}-\frac{p^2m_2^2\phi}{4d_2^2\xi})\nonumber\\
&&e((d_1d_2)^2\xi\phi c)\}\tilde{h}(\frac{\xi_1m_1}{d_1},l_1,\tau_1)\tilde{h}(\frac{p\xi_2m_2}{d_2},l_2,\tau_2)d\tau_1d\tau_2\frac{d\xi_1d\xi_2}{\xi_1^{3/2-k_1}\xi_2^{3/2-k_2}}\nonumber\\
&&+\sum_{l_1=0}^{k_1}\sum_{l_2=0}^{k_2}\sum_{\substack{d_1|m_1\\
d_2|m_2/p}}\sum_{c\geq 1}\int_0^{\infty}\int_0^{\infty}\int_0^{1}\int_0^{1}\Im
\{\frac{S_c\zeta_8}{c^{\frac{5}{2}}}e_c(\frac{m_1m_2}{2pd_1d_2}-\frac{m_1^2\xi}{4d_1^2\phi}-\frac{m_2^2\phi}{4p^2d_2^2\xi})\nonumber\\
&&e((d_1d_2)^2\xi\phi c)\}\tilde{h}(\frac{\xi_1m_1}{d_1},l_1,\tau_1)\tilde{h}(\frac{m_2\xi_2/p}{d_2},l_2,\tau_2)d\tau_1d\tau_2\frac{d\xi_1d\xi_2}{\xi_1^{3/2-k_1}\xi_2^{3/2-k_2}}\nonumber
\end{eqnarray}

According to whether or not $p|(c,\ast,\ast)$ in $S_{c,\ast,\ast}$,
we can decompose the above sums $I_1$, $I_2$, $II_1$, $II_2$ into
the following 8 terms$$I_1=I_{11}+I_{12}, \quad I_2=I_{21}+I_{22},
\quad II_1=II_{11}+II_{12}, \quad II_2=II_{21}+II_{22}.$$ Note if
$p|(c,\ast,\ast)$, $S_{c,\ast,\ast}=0$ unless $p^2|c$. Let
$c=p^2c_1$, we have
$$S_{c,\frac{|m_1p|}{d_1},\frac{|m_2|}{d_2}}=S_{c_1,\frac{|m_1|}{d_1},\frac{|m_2|}{pd_2}}p^2(1-\frac{\delta(p,c_1)}{p}),$$
where $\delta(p,c_1)=0$ if $p|c_1$; $\delta(p,c_1)=1$ if $p\nmid
c_1$. Hence we can write $I_{11}=I'_{11}-I''_{11}$ correspondingly.

Similarly we have
$$S_{c,\frac{|m_1|}{pd_1},\frac{|m_2|}{d_2}}=S_{c_1,\frac{|m_1|}{p^2d_1},\frac{|m_2|}{pd_2}}p^2(1-\frac{\delta(p,c_1)}{p}),$$
and write $I_{21}=I'_{21}-I''_{21}$,
$$S_{c,\frac{|m_1|}{d_1},\frac{|m_2p|}{d_2}}=S_{c_1,\frac{|m_1|}{pd_1},\frac{|m_2|}{d_2}}p^2(1-\frac{\delta(p,c_1)}{p}),$$
and write $II_{11}=II'_{11}-II''_{11}$,
$$S_{c,\frac{|m_1|}{d_1},\frac{|m_2|}{pd_2}}=S_{c_1,\frac{|m_1|}{pd_1},\frac{|m_2|}{p^2d_2}}p^2(1-\frac{\delta(p,c_1)}{p}),$$
and write $II_{21}=II'_{21}-II''_{21}$ corresponding $p|c_1$ or not.

By the induction hypothesis on $(\frac{m_1}{p},\frac{m_2}{p})$, we
have $I'_{11}+I'_{21}=II'_{11}+II'_{21}.$

We have $S_{cp,a,b}=p^2S_{c,a,b}$ and $S_{tp^2,ap,b}=0$ if $p\nmid
bc$. Using this and the evaluation of $S_{c,a,b}$ we can verify that
$$I_{12}(p|d_1)=II_{12}(p|d_2),$$
where $I_{12}(p|d_1)$ means the partial sum of $I_{12}$ in which
$p|d_1$. Similarly, we have
$$I_{12}(p\nmid d_1,p\nmid d_2,p\nmid c)=II_{12}(p\nmid d_2,p\nmid d_1,p\nmid c),$$
$$I_{12}(p\nmid d_1,p\parallel d_2,p\nmid c)=II_{12}(p\nmid d_2,p\parallel d_1,p\nmid c),$$
$$I_{12}(p\nmid d_1,p^2| d_2,p\nmid c)=I''_{11}(p\nmid d_1,p^2|m_2/d_1),$$
$$I_{12}(p\nmid d_1,p^2| d_2,p\nmid c)=I''_{11}(p\nmid d_1,p\parallel m_2/d_2),$$
$$II''_{11}(p\nmid d_2,p^2|m_1/d_1)=II_{12}(p\nmid d_2,p^2| d_1,p\nmid c),$$
$$II''_{11}(p\nmid d_2,p\parallel m_1/d_1)=II_{12}(p\nmid d_2,p\nmid c),$$
$$I''_{11}(p| d_1)=II''_{11}(p|d_2),$$
$$I_{22}(p|d_2)=II_{22}(p|d_1),$$
$$I_{22}(p\nmid d_2,p\nmid d_1,p\nmid c)=II_{22}(p\nmid d_1,p\nmid d_2,p\nmid c),$$
$$I_{22}(p\nmid d_2,p\parallel d_1,p\nmid c)=II_{22}(p\nmid d_1,p\parallel d_2,p\nmid c),$$
$$I_{22}(p\nmid d_2,p^2| d_1,p\nmid c)=I''_{21}(p\nmid d_2,p^3|m_1/d_1),$$
$$I_{22}(p\nmid d_2,p| c)=I''_{21}(p\nmid d_2,p^2\parallel m_1/d_1),$$
$$II''_{21}(p| d_1)=I''_{21}(p| d_2),$$
$$II_{22}(p\nmid d_1,p^2|d_2,p\nmid c)=II''_{21}(p\nmid d_1,p^3|m_2/d_2),$$
$$II_{22}(p\nmid d_1,p| c)=II''_{21}(p\nmid d_1,p^2\parallel m_2/d_2).$$
Hence we deduce from the above identities that
$$Q_{ND}(T_pP_{m_1,k_1},P_{m_2,k_2})=Q_{ND}(P_{m_1,k_1},T_pP_{m_2,k_2}).$$
This completes the proof of
$$Q(T_pP_{m_1,k_1},P_{m_2,k_2})=Q(P_{m_1,k_1},T_pP_{m_2,k_2})$$ for
each $T_p$, $p$ is a prime. 
\end{proof}

For case (b), we need consider $Q(P_{m_1,k_1}, P_{h,m_2})$ and analyze the self-adjointness with Hecke operator in this case. Using the formula of $<Op(P_{m,k})\phi_j,\phi_j>$ which we just evaluated  above and the formula of $<Op(P_{h,m})\phi_j, \phi_j>$  in \cite{Z}, we have 

\begin{IEEEeqnarray*}{rCl}
&&Q(P_{m_1,k_1}, P_{h,m_2})
\nonumber\\
&=&
\sum_{\frac{m_1}{d_1}=\frac{m_2}{d_2}}\int_0^{\infty}\int_0^1
\cos(\frac{\pi
m_1}{d_1}\xi(2\tau-1))\exp\left(\frac{-m_1\xi\sqrt{\tau(1-\tau)}}{
d_1}\right)(\tau(1-\tau))^{k_1}d\tau \nonumber\\
&&\int_0^1 \frac{\cos(\frac{\pi m_2}{d_2}\xi(2\eta-1))
h\left(\frac{\xi\sqrt{\eta(1-\eta)}}{
d_2}\right)}{\eta(1-\eta)}d\eta
\frac{\xi^{k_1}d\xi}{\xi^2}+\\&&\sum_{\substack{d_1|m_1\\
d_2|m_2}}\sum_{c\geq 1} \Im
(\frac{S_c\zeta_8}{c^{\frac{3}{2}}}\int\int_{\mathrm{R}^2}e_c(\frac{m_1m_2}{2d_1d_2}-\frac{m_1^2\xi}{4d_1^2\phi}-\frac{m_2^2\phi}{4d_2^2\xi})\frac{2^{\frac32}}{(\xi\phi)^{\frac32}}\nonumber\\
& &e((d_1d_2)^2\xi\phi c) \int_0^1\int_0^1 \frac{\cos(\pi
m_1d_2\xi(2\tau-1))\cos(\pi
m_2d_1\phi(2\eta-1))(\tau(1-\tau))^{k_1}}{\eta(1-\eta)}\nonumber\\
& &\exp(-m_1\xi d_2\sqrt{\tau(1-\tau)})h(\phi
d_1\sqrt{\eta(1-\eta)})d\tau
d\eta d\xi d\phi)
\end{IEEEeqnarray*}
Note that $P_{h,m}$ is a weight $0$ Poincar\'{e} series and  under the Hecke operator, we have
$$T_nP_{h,m}(z)=\sum_{d|(m,n)}(\frac{d^2}{n})^{\frac12}P_{h(\frac{ny}{d^2}),\frac{mn}{d^2}}(z).$$
A similar argument about the self-adjointness with respect to Hecke operator  works for $Q(P_{m_1,k_1}, P_{h,m_2})$, i.e.
$$Q(T_pP_{m_1,k_1}, P_{h,m_2})=Q(P_{m_1,k_1}, T_pP_{h,m_2}).$$

For case (c) of $\phi_j$ and $\phi_k$ both being Maass forms, it was shown in \cite{Z}.  
Thus, combining these three cases, the Hermitian form $Q(\cdot,\cdot)$
defined on the space spanned by $P_{m,k}$'s is self-adjoint with
respect to the Hecke operators $T_n$, $n\geq 1$. Hence, for the   generating vectors $\phi_j,\phi_k$ of  each irreducible representation, we obtain
\begin{proposition}
 $$Q(T_n\phi_j,\phi_k)=Q(\phi_j,T_n\phi_k)$$ if $\phi_j,\phi_k$ is either weight $k$ holomorphic form or Maass form. \end{proposition}

From  this, we have 
$$\lambda_n(\phi_j)Q(\phi_j,\phi_k)=\lambda_n(\phi_k)Q(\phi_j,\phi_k).$$ Since there is an $n$ such that $\lambda_n(\phi_j)\neq\lambda_n(\phi_k)$ if $\phi_j$, $\phi_k$ are  generator vectors of two distinct irreducible representations, 
 we  deduce the orthogonality,  $Q(\phi_j,\phi_k)=0$ if  $\phi_j$, $\phi_k$ are  in distinct eigenspaces of the orthogonal decomposition $(21)$.

In the next section we calculate the eigenvalue of $B$ on such a generating
Maass-Hecke cusp form.

\section{Eigenvalue of $Q$} 
\vspace{.8mm}
  In this section, we shall evaluate the weighted quantum variance on each eigenspace $U_{\pi_j^k}$ by applying Woodbury's explicit formula for the Ichino's trilinear formula with special vectors (see Appendix A), Rankin-Selberg
theory, Kuznetsov formula and a principle observed in Luo-Rudnick-Sarnak (Remark 1.4.3 and Prop. 3.1 in \cite{Luo0}). 
\begin{proposition}
    For weight $k$ holomorphic Hecke eigenform $f$  with $\|f\|_2=1$, 
    we have
\begin{eqnarray}
&&\lim_{T\rightarrow\infty}\frac1T\sum_{j\ge 1}u\left(\frac{t_j}{T}\right)L (1,\mathrm{sym}^2\varphi_j)|\omega_j(f)|^2= 2^{k-1}\frac{\Gamma^2(\frac k2)}{\Gamma(k)}L(\frac12,f).\nonumber
\end{eqnarray}
\vspace{.5mm}
\end{proposition}
\begin{proof}
Let $\Lambda(s,\varphi_j)$ be the associated completed $L$-function of $\varphi_j$, which
admits analytic continuation to the whole complex plane and
satisfies the functional equation:
$$\Lambda(s,\varphi_j):=\pi^{-s}\Gamma\left(\frac{s+it_{j}}{2}\right)\Gamma\left(\frac{s-it_{j}}{2}\right)L(s,\varphi_j)=\Lambda(1-s,\varphi_j).$$
Moreover, we have
$$\Lambda(s,\mathrm{sym}^2(\varphi_j))=\pi^{-3s/2}\Gamma\left(\frac s2\right)\Gamma\left(\frac{s}{2}+it_{j}\right)\Gamma\left(\frac{s}{2}-it_{j}\right)L(s,\mathrm{sym}^2\varphi_j).$$

For weight $k$ holomorphic Hecke eigenform $f$, we have the associated completed $L$-function, $$\Lambda(s,f):=\pi^{-s}\Gamma\left(\frac{s+\frac{k-1}{2}}{2}\right)\Gamma\left(\frac{s+\frac{k+1}{2}}{2}\right)L(s,f).$$

Thus, we obtain  the  Rankin-Selberg $L$-function,

\begin{eqnarray}
\Lambda(s, f\otimes\mathrm{sym}^2\varphi_j)&=&\pi^{-3s}\Gamma\left(\frac{s+\frac{k-1}{2}}{2}\right)\Gamma\left(\frac{s+\frac{k-1}{2}}{2}+it_j\right)\Gamma\left(\frac{s+\frac{k-1}{2}}{2}-it_j\right)\nonumber\\
&&\Gamma\left(\frac{s+\frac{k+1}{2}}{2}\right)\Gamma\left(\frac{s+\frac{k+1}{2}}{2}+it_j\right)\Gamma\left(\frac{s+\frac{k+1}{2}}{2}-it_j\right)\nonumber\\
&&L(s,f\otimes\mathrm{sym}^2\varphi_j),\nonumber
\end{eqnarray}

By Ichino's general trilinear formula \cite{ichino} and its explication in the Appendix with the explicit vectors at hand, we can express the  triple product integrals of eigenforms in terms of  the  Rankin-Selberg $L$-function $\Lambda(s, f\otimes\mathrm{sym}^2\varphi_j)$ as follows;

\begin{eqnarray}
|<Op(f)\varphi_j,\varphi_j>|^2=\frac{1}{2^4}\cdot\frac{\Lambda(\frac12,f\otimes\varphi_j\otimes\varphi_j)}{\Lambda(1,\mathrm{sym}^2\varphi_j)^2\Lambda(1,\mathrm{sym}^2f)}\cdot\frac{2^{k-1}\pi^k}{(\frac{1}{2}+it_j)_{\frac{k}{2}}(\frac{1}{2}-it_j)_{\frac{k}{2}}}\nonumber
\end{eqnarray}
where $(z)_m=z(z+1)\cdots(z+m-1)$ . 
The local factors at $\infty$ place is (Lemma 8 in Woodbury's calculation),
$$\zeta_{\mathbb{R}}(2)^2\cdot\frac{L_{\infty}(\frac12,f\otimes\varphi_j\otimes\varphi_j)}{L_{\infty}(1,\mathrm{sym}^2\varphi_j)^2L_{\infty}(1,\mathrm{sym}^2f)}=\frac{|\Gamma(\frac{k}{2}+2it_j)|^2|\Gamma(\frac{k}{2})|^2}{2^{k-3}\pi^{k-1}}\Gamma(k)|\Gamma(\frac12+it_j)|^4$$


By Stirling formula and the duplication formula of  the Gamma factors, it amounts to
\begin{eqnarray}
&&|<Op(f)\varphi_j,\varphi_j>|^2\nonumber\\
&=&
\frac{L(\frac12,f)L(\frac12,f\otimes\mathrm{sym}^2(\varphi_j))|\Gamma(\frac k2)|^2|a_{j}(1)|^2}{ 4\pi^{-1}t_j\cosh \pi t_jL(1,\mathrm{sym}^2\varphi_j)L(1,\mathrm{sym}^2f)}(1+O(t_{j}^{-1}))\nonumber
\end{eqnarray}

where $a_j(n)$ is the $n$-th Fourier coefficient of $\varphi_j$ with $\|\varphi_j\|_2=1$ and  $$|a_j(1)|^2=\frac{2\cosh \pi t_j}{L(1,\textrm{sym}^2\varphi_j)},$$ 

Next we apply the approximate functional equation of $L(s,f\otimes\mathrm{sym}^2\varphi_j)$, and Kuznetsov formula to evaluate the variance sum in the Proposition. We compute $$\sum_{j\geq
1}u\left(\frac{t_j}{T}\right)L(1,\mathrm{sym}^2\phi_j)|<Op(f)\varphi_j,\varphi_j>|^2$$

Let $\Phi$ be the cuspidal automorphic form on $GL(3)$ which is the
Gelbart-Jacquet lift of the cusp form $\phi$, with the Fourier
coefficients $a_{\Phi}(m_1,m_2)$ \cite{Bu}, where
$$a_{\Phi}(m_1,m_2)=\sum_{d|(m_1,m_2)}\lambda_{\Phi}(\frac{m_1}{d},1)\lambda_{\Phi}(\frac{m_2}{d},1)\mu(d),$$
and $$\lambda_{\Phi}(r,1)=\sum_{s^2t=r}\lambda_{\phi}(t^2).$$ The
Rankin-Selberg convolution $L(s,f\otimes\mathrm{sym}^2\varphi_j)$
is represented by the Dirichlet series,
$$L(s,f\otimes\mathrm{sym}^2\varphi_j)=\sum_{m_1,m_2\ge 1}\lambda_{f}(m_1)a_{\Phi_j}(m_1,m_2)(m_1m_2^2)^{-s},$$
where $\lambda_{f}(r)$ is the $r$-th Hecke eigenvalue of
$f$.\\
Since
$$\Lambda(1/2,f\otimes\mathrm{sym}^2\varphi)=\frac{1}{\pi i}\int_{(2)}\Lambda(s+1/2,f\otimes\mathrm{sym}^2\varphi)\frac{ds}{s}.$$
we have the following approximate functional
equation,$$L(1/2,f\otimes\mathrm{sym}^2\varphi_j)=2\sum_{m_1,m_2\ge
1}\lambda_{f}(m_1)a_{\Phi_j}(m_1,m_2)(m_1m_2^2)^{-1/2}V_{t_j}(m_1m_2^2)$$
where
\begin{eqnarray}V_{t_j}(y)&=&\frac{1}{2\pi i}\int_{(2)}y^{-s}\frac{\gamma(1/2+s,f\otimes\mathrm{sym}^2\varphi_j)}{\gamma(1/2,f\otimes\mathrm{sym}^2\varphi_j)}\frac{ds}{s}\nonumber\\
&=&\frac{1}{2\pi i}\int_{(2)}(1+P_{t_j}(s))\frac{\Gamma\left(\frac{s+\frac{k+1}{2}}{2}\right)\Gamma\left(\frac{s+\frac{k-1}{2}}{2}\right)}{\Gamma\left(\frac k4+\frac 14\right)\Gamma\left(\frac k4-\frac14\right)}\left(\frac{y}{t_j^2}\right)^{-s}\frac{ds}{s}\nonumber
\end{eqnarray}
where $$P_t(s)=\sum_{1\le r\le
10}\frac{p_{r+1}(s)}{t^r}+O(\frac{|s|^{12}}{t^{11}})$$ is an analytic
function in $\mathcal{R}s\ge -2$. $p_{r}(s)$ is a polynomial of
degree at most $r$  and independent of $t$. And the Gamma factor is
\begin{eqnarray}
\gamma(s,f\otimes\mathrm{sym}^2\varphi_j)&=&\pi^{-3s}\Gamma\left(\frac{s+\frac{k-1}{2}}{2}\right)\Gamma\left(\frac{s+\frac{k-1}{2}}{2}+it_j\right)\Gamma\left(\frac{s+\frac{k-1}{2}}{2}-it_j\right)\nonumber\\
&&\Gamma\left(\frac{s+\frac{k+1}{2}}{2}\right)\Gamma\left(\frac{s+\frac{k+1}{2}}{2}+it_j\right)\Gamma\left(\frac{s+\frac{k+1}{2}}{2}-it_j\right)\nonumber
\end{eqnarray}
Thus, by writing $$\Gamma(\frac{k+1}{2}+it)=\Gamma(\frac12+it)(\frac12+it)_{\frac k2}$$ and duplication formula of Gamma functions, we have
\begin{eqnarray}
&&\sum_{j\geq
1}u\left(\frac{t_j}{T}\right)L(1,\mathrm{sym}^2\varphi_j)|<Op(f)\varphi_j,\varphi_j>|^2\nonumber\\
&=& \frac{\pi}{4}L(\frac12,f)|\Gamma (\frac k2)|^2\sum_{t_j\geq
1}u\left(\frac{t_j}{T}\right)\frac{|a_j(1)|^2}{\cosh \pi t_j}L(1/2,f\otimes\mathrm{sym}^2(\phi_j))+O(\log T)\nonumber\\
&=& \frac{\pi}{4}L(\frac12,f)|\Gamma (\frac k2)|^2\sum_{t_j\geq
1}u\left(\frac{t_j}{T}\right) \frac{|a_j(1)|^2}{\cosh \pi t_j}\nonumber\\
&&\sum_{m_1,m_2\ge
1}\lambda_{f}(m_1)a_{\Phi_j}(m_1,m_2)(m_1m_2^2)^{-1/2}V_{t_j}(m_1m_2^2)+O(\log T)\nonumber
\\
&=& \frac{\pi}{4}L(\frac12,f)|\Gamma (\frac k2)|^2\sum_{t_j\geq
1}\sum_{d\geq 1}\frac{\mu(d)}{d^{\frac32}}\sum_{n_1,n_2\geq
1}\lambda_{f}(dn_1)V_{t_j}(d^3n_1n_2^2)(n_1n_2^2)^{-1/2}\nonumber\\
&&u\left(\frac{t_j}{T}\right)\frac{|a_j(1)|^2}{\cosh \pi t_j}\lambda_{\Phi_j}(n_1,1)\lambda_{\Phi_j}(n_2,1)+O(\log T)\nonumber
\end{eqnarray}
\begin{eqnarray}
&=& \frac{\pi}{4}L(\frac12,f)|\Gamma (\frac k2)|^2\sum_{t_j\geq
1}\sum_{d\geq
1}\frac{\mu(d)}{d^{\frac32}}\sum_{s_1,s_2,w_1,w_2\geq
1}\lambda_{f}(ds_1^2w_1)V_{t_j}(d^3s_1^2w_1s_2^4w_2^2)(s_1^2w_1s_2^4w_2^2)^{-1/2}\nonumber\\
&&u\left(\frac{t_j}{T}\right)\frac{|a_j(1)|^2}{\cosh \pi t_j}\lambda_j(w_1^2)\lambda_j(w_2^2)+O(\log T)\nonumber
\\
&=& \frac{\pi}{4}L(\frac12,f)|\Gamma (\frac k2)|^2\sum_{d\geq
1}\frac{\mu(d)}{d^{\frac32}}\sum_{s_1,s_2,w_1,w_2\geq
1}\lambda_{f}(ds_1^2w_1)(s_1^2w_1s_2^4w_2^2)^{-1/2}\nonumber\\
&&\sum_{t_j\geq
1}V_{t_j}(d^3s_1^2w_1s_2^4w_2^2)u\left(\frac{t_j}{T}\right)\frac{|a_j(1)|^2}{\cosh \pi t_j}\lambda_j(w_1^2)\lambda_j(t_2^2)+O(\log T)\nonumber
\end{eqnarray}
For the inner sum, by the Kuznetsov formula, we have
\begin{eqnarray}
&&\sum_{t_j\geq
1}V_{t_j}(d^3s_1^2w_1s_2^4w_2^2)u\left(\frac{t_j}{T}\right) \frac{|a_j(1)|^2}{\cosh \pi t_j} \lambda_j(w_1^2)\lambda_j(w_2^2)\nonumber\\
&=&\frac{\delta(w_1,w_2)}{\pi^2}\int_{-\infty}^{\infty}V_t(d^3s_1^2s_2^4w_1w_2^2)u\left(\frac{t}{T}\right)t\tanh(\pi t)dt\nonumber
\\
&&-\frac{2}{\pi}\int_0^{\infty}V_t(d^3s_1^2w_1s_2^4w_2^2)\frac{u\left(\frac{t}{T}\right)}{|\zeta(1+2it)|^2}d_{it}(w_1^2)d_{it}(w_2^2)dt\nonumber\\
&&+\frac{2i}{\pi}\sum_{c\ge
1}\frac{S(w_1^2,w_2^2;c)}{c}\int_{-\infty}^{\infty}J_{2it}(\frac{4\pi
t_1t_2}{c})V_t(d^3s_1^2s_2^4w_1w_2^2) u\left(\frac{t}{T}\right)\frac{tdt}{\cosh(\pi
t)}\nonumber
\end{eqnarray}

 We will estimate the above three sums respectively.



In the diagonal term, let $w=dw_1=dw_2$,  
\begin{eqnarray}
&&\sum_{d\geq
1}\frac{\mu(d)}{d^{\frac32}}\sum_{s_1,s_2,w_1=w_2\geq
1}\lambda_{f}(ds_1^2w_1)(s_1^2w_1s_2^4w_2^2)^{-1/2}V_{t}(d^3s_1^2w_1s_2^4w_2^2)\nonumber\\
&=&\sum_{w\geq 1}w^{-\frac 32}\sum_{d|w}\mu(d)\sum_{s_2\ge
1}s_2^{-2} \sum_{s_1\ge
1}s_1^{-1}\lambda_{f}(ws_1^2)V_t(w^3s_1^2s_2^4)\nonumber\\
&=&\sum_{s_2\ge
1}s_2^{-2} \sum_{s_1\ge
1}s_1^{-1}\lambda_{f}(s_1^2)V_t(s_1^2s_2^4)\nonumber
\end{eqnarray}
%
The diagonal term  is
\begin{eqnarray}
 \frac{\pi}{4}L(\frac12,f)|\Gamma (\frac k2)|^2\int_{-\infty}^{\infty}u\left(\frac{t}{T}\right)\sum_{s_2\ge
1}s_2^{-2} \sum_{s_1\ge
1}s_1^{-1}\lambda_{f}(s_1^2)V_t(s_1^2s_2^4)\tanh(\pi
t)tdt\nonumber
\end{eqnarray}
For the sum over $s_1$, we have $$\sum_{s_1\ge
1}s_1^{-1}\lambda_{f}(s_1^2)V_t(s_1^2s_2^4)=\frac{1}{2\pi
i}\int_{(2)}\sum_{s_1\ge
1}\frac{\lambda_{f}(s_1^2)}{s_1^{2s+1}}U_t(s)(\frac{s_2^4}{t^2})^{-s}\frac{ds}{s},$$
where
$$U_t(s)=(1+P_t(s))\frac{\Gamma\left(\frac{s+\frac{k+1}{2}}{2}\right)\Gamma\left(\frac{s+\frac{k-1}{2}}{2}\right)}{\Gamma\left(\frac k4+\frac 14\right)\Gamma\left(\frac k4-\frac14\right)},$$
and$$P_t(s)=\sum_{1\le r\le
N}\frac{p_{r+1}(s)}{t^r}+O(\frac{|s|^{N+2}}{t^{N+1}})$$ is an analytic
function in $\mathcal{R}s\ge -2$. $p_{r+1}(s)$ is a polynomial of
degree at most $r+1$. \\
Also, we have $$\sum_{s_1\ge
1}\frac{\lambda_{f}(s_1^2)}{s_1^{s}}=\frac{1}{\zeta(2s)}L(s,\mathrm{sym}^2 f).$$
Thus, moving the line of integration in the sum over $s_1$ to
$\mathcal{R}(s)=-1/4+\epsilon$, we get
$$\sum_{s_1\ge
1}s_1^{-1}\lambda_{f}(s_1^2)V_t(s_1^2s_2^4)=\frac{1}{\zeta(2)}L(1,\mathrm{sym}^2 f)+O(T^{-1/2+\epsilon}).$$
Therefore, we get the diagonal terms contribute
$$\frac{\pi}{4}TL(1,\mathrm{sym}^2 f)L(\frac12,f)|\Gamma (\frac k2)|^2+O(T^{1/2+\epsilon}).$$

Since $$V_t(y)=\frac{1}{2\pi i}\int_{(2)}U_t(s)\left(\frac{y}{t^2}\right)^{-s}\frac{ds}{s},$$
$V_t(y)$ can be written as 
$$V_t(y)=V\left(\frac{y}{t^2}\right)+\sum_{1\le r\le N}\frac{1}{t^r}V_r\left(\frac{y}{t^2}\right)+O(\frac{1}{t^{N+1}}).$$
Thus,  the non-diagonal terms are 
\begin{eqnarray}
&&\sum_{d\geq
1}\frac{\mu(d)}{d^{\frac32}}\sum_{s_1,s_2,t_1,t_2\geq
1}\lambda_{f}(ds_1^2t_1)(s_1^2t_1s_2^4t_2^2)^{-1/2}\sum_{c\geq
1}\frac{S(t_1^2,t_2^2;c)}{c}\nonumber\\
&&\int_{-\infty}^{\infty}J_{2it}(\frac{4\pi
t_1t_2}{c})V(\frac{d^3s_1^2t_1s_2^4t_2^2}{t^2})u\left(\frac{t}{T}\right)\frac{tdt}{\cosh(\pi
t)}\nonumber
\end{eqnarray}
  Let $x=\frac{4\pi t_1t_2}{c}$, the inner integral in the
non-diagonal terms
is$$\frac12\int_{-\infty}^{\infty}\frac{J_{2it}(x)-J_{-2it}(x)}{\sinh{\pi
t}}V(\frac{d^3s_1^2t_1s_2^4t_2^2}{t^2})u\left(\frac{t}{T}\right)\tanh(\pi
t)tdt.$$ Since $\tanh (\pi t)=$sgn$ (t)+O(e^{-\pi |t|})$ for large
$|t|$, we can remove $\tanh(\pi t)$ by getting a negligible term
$O(T^{-N})$ for any $N>0$. Applying the Parseval identity, the
Fourier transform in  \cite{Ba},
$$\left(\frac{\widehat{J_{2it}(x)-J_{-2it}(x)}}{\sinh(\pi
t)}\right)(y)=-i\cos(x\cosh(\pi y)).$$ and the evaluation of the
Fresnel integrals, the integral is
\begin{eqnarray}
&&\frac12\int_{-\infty}^{\infty}(\frac{J_{2it}(x)-J_{-2it}(x)}{\sinh{\pi
t}})^{\wedge}(y)(V(\frac{d^3s_1^2t_1s_2^4t_2^2}{t^2})u\left(\frac{t}{T}\right)t)^{\wedge}(y)dy\nonumber\\
&=&\frac{-i}{2}\int_{-\infty}^{\infty}(\cos(x\cosh(\pi y)))(V(\frac{d^3s_1^2t_1s_2^4t_2^2}{t^2})u\left(\frac{t}{T}\right)t)^{\wedge}(y)dy\nonumber\\
&=&\frac{-i}{2}\int_{-\infty}^{\infty}(\cos(x+\frac12\pi^2xy^2))(V(\frac{d^3s_1^2t_1s_2^4t_2^2}{t^2})u\left(\frac{t}{T}\right)t)^{\wedge}(y)dy\nonumber\\
&=&\frac{-i}{2}\int_{0}^{\infty}(\cos(x-y+\frac{\pi}{4}))(V(\frac{d^3s_1^2t_1s_2^4t_2^2}{t^2})u\left(\frac{t}{T}\right))(\sqrt{\frac{xy}{2}})\frac{dy}{\sqrt{\pi y}}\nonumber\\
&=&\frac{-i}{2}\int_{0}^{\infty}(\cos(x-y+\frac{\pi}{4}))V(\frac{2d^3s_1^2t_1s_2^4t_2^2}{xy})u\left(\frac{\sqrt{\frac{xy}{2}}}{T}\right)\frac{dy}{\sqrt{\pi y}}\nonumber\\
&=&\frac{-i}{2}\int_{0}^{\infty}(\cos(4\pi t_1t_2c^{-1}-y+\frac{\pi}{4}))V(\frac{2d^3s_1^2t_1s_2^4t_2^2}{4\pi t_1t_2c^{-1}y})u\left(\frac{\sqrt{\frac{4\pi t_1t_2c^{-1}y}{2}}}{T}\right)\frac{dy}{\sqrt{\pi
y}}\nonumber
\end{eqnarray}
Note: Here all the equation is up to an error of $O(xT^{-4})$. The higher Taylor coefficients of the $\cosh$ factor is negligible by studying the stationary phases as did in \cite{LY}.

 Thus, the
non-diagonal terms is concentrated on
$$T^{2-\epsilon}\ll t_1t_2c^{-1}y\ll T^2.$$
So, we can assume $d^3s_1^2t_1s_2^4t_2^2\ll T^{2+\epsilon}$ since
$V(\xi)$ has exponential decay as $\xi\rightarrow\infty.$ By
partial integration, the terms with $c\gg T^{\epsilon}$ and  also the terms 
$t_1t_2\ll T^{2-4\epsilon}$ contribute $O(1)$. So we can assume
$c\ll T^{\epsilon}$ and $t_1t_2\gg T^{2-4\epsilon}$, we also have $t_1t_2^2\ll T^{2+\epsilon}$ therefore we
have $t_2\ll T^{5\epsilon}$, also we have the sum over $s_1$ and
$s_2$ converges. Let $t=\frac{\sqrt{2\pi t_1t_2c^{-1}y}}{T}$, the
inner integral is
\begin{eqnarray}\frac{T\sqrt{2c}}{\pi\sqrt{ t_1t_2}}\int_0^{\infty}u(t)(\cos(4\pi t_1t_2c^{-1}-(tT)^2c/(2\pi
t_1t_2)+\frac{\pi}{4}))V(\frac{2d^3s_1^2t_1s_2^4t_2^2}{t^2T^2})dt\nonumber
\end{eqnarray} From Hecke's bound
$$\sum_{r\leq R}\lambda_{f}(r)e(\alpha r)r^{-1/2}\ll_{\epsilon}
R^{\epsilon},$$ where $\alpha\in \R$ and the Hecke
relation $$\lambda_{f}(r_1r_2)=\sum_{d|(r_1,
r_2)}\mu(d)\lambda_{f}(r_1/d)(r_2/d);$$ and partial summation,
we get the non-diagonal terms contribute $O(T^{5\epsilon})$.

To evaluate the continuous part, we need rewrite
\begin{eqnarray}
&&\sum_{d\geq
1}\frac{\mu(d)}{d^{\frac32}}\sum_{s_1,s_2,t_1,t_2\geq
1}\lambda_{f}(ds_1^2t_1)(s_1^2t_1s_2^4t_2^2)^{-1/2}\nonumber\\
&&\int_{0}^{\infty}V(\frac{d^3s_1^2t_1s_2^4t_2^2}{t^2})\frac{u\left(\frac{t}{T}\right)}{|\zeta(1+2it)|^2}d_{it}(t_1^2)d_{it}(t_2^2)dt\nonumber
\end{eqnarray}
with respect to $L$-function and we obtain the continuous part contributes
\begin{eqnarray}
&&\int_0^{\infty}u\left(\frac{t}{T}\right)\frac{1}{|\zeta(1+2it)|^2}|L(\frac12+2it,f)|^2\nonumber\\
&&\frac{|\Gamma(\frac14-\frac{k}{2}-it)\Gamma(\frac14+\frac{k}{2}-it)|^2}{|\Gamma(\frac12+it)|^4}dt\nonumber
\end{eqnarray}
By Stirling formula and the Jutila's bound the subconvex bound \cite{jut},
$$L(\tfrac{1}{2}+it,f_j)\ll (\kappa_j+t)^{1/3+\epsilon},$$ we
obtain the continuous part contributes $O(T^{\frac23+\epsilon}).$
\newline

So we conclude that
\begin{eqnarray}
&&\sum_{j\geq
1}u\left(\frac{t_j}{T}\right)L(1,\mathrm{sym}^2\phi_j)|<Op(f)\varphi_j,\varphi_j>|^2\nonumber\\
&=&\frac{1}{2^k\pi^{k+1}}TL(1,\mathrm{sym}^2f)L(\frac12,f)|\Gamma
(\frac k2)|^2+O(T^{2/3+\epsilon}).
\end{eqnarray}
Since we normalize $f$, such that $<f,f>=1$ and from the fact $$|a_f(1)|^{-2}=2^{1-2k}\pi^{-k-1}\Gamma(k)L(1,\mathrm{sym}^2f),$$
we obtain the eigenvalue of $B$ at $f$ is
$$L(\frac12,f)\frac{2^{k-1}|\Gamma(\frac k2)|^2}{\Gamma(k)}.$$

Therefore, we complete the proof of the Proposition 4.
\end{proof}
Moreover from \cite{Z},  we have the following weighted quantum variance for Maass forms,
\vspace{.6mm}
\begin{proposition}
Let $\phi(z)$ be an even Maass-Hecke cuspidal
eigenform for  $\Gamma$, with the Laplacian
eigenvalue $\lambda_{\phi}=\frac14+t_{\phi}^2$, 
we have
\begin{eqnarray}
&&\lim_{T\rightarrow\infty}\frac1T\sum_{j\ge 1}u\left(\frac{t_j}{T}\right)L (1,\mathrm{sym}^2\phi)|<Op(\phi)\varphi_j,\varphi_j>|^2\nonumber\\
&=&L(\frac12,\phi)\frac{|\Gamma(\frac14-\frac{it_{\phi}}{2})|^4}{2\pi|\Gamma(\frac12-it_{\psi})|^2}.\nonumber
\end{eqnarray}

\end{proposition}
Note: In \cite{Z}, although the averaging there is against a specific weight function, it can be  removed using the same technique as we remove the weight $u(t)$ in next section.
 
Next, we will remove the weights in Proposition 4 and Proposition 5.

\section{ Removing the Weights}

We turn to removing the arithmetic weight $L (1,\mathrm{sym}^2\varphi_j)$ in our main Theorem 1.  We focus here on calculating the modified diagonal terms since the modified off-diagonal terms have the analogous estimates. 

We have 
\begin{eqnarray}
L (s,\mathrm{sym}^2\varphi_j)=\prod_p(1-\alpha_j^2(p)p^{-s})^{-1}(1-\beta_j^2(p)p^{-s})^{-1}(1-p^{-s})^{-1}\nonumber
\end{eqnarray}
where 
\begin{eqnarray}
\lambda_j(p)=\alpha_j(p)+\beta_j(p), \textrm{ and } \alpha_j(p)\beta_j(p)=1\nonumber
\end{eqnarray}
Hence from the Hecke relations $$\lambda_j^2(p)=\lambda_j(p^2)+1,$$
we have 
\begin{eqnarray}
\frac{1}{L (s,\mathrm{sym}^2\varphi_j)}&=&\prod_p(1-\lambda_j(p^2)p^{-s}+\lambda_j(p^2)p^{-2s}-p^{-3s})\nonumber\\
&:=& \sum_{n=1}^{\infty}\mu_{\mathrm{sym}^2\varphi_j}(n)n^{-s}\nonumber
\end{eqnarray}
where 
$$\mu_{\mathrm{sym}^2\varphi_j}(n)=\sum_{\substack{ab^2c^3=n\\(a,b)=(b,c)=(a,c)=1}}\mu(a)\lambda_j(a^2)\mu^2(b)\lambda_j(b^2)\mu(c).$$
Note that $$\mu_{\mathrm{sym}^2\varphi_j}(n)\ll_{\epsilon}n^{\epsilon}(|\lambda_j(n)|^4+1).$$
Hence  it follows that for $t_j\le R$ and $\xi\ge 1$,   \cite{GRS} and \cite{Iw90}, 
\begin{eqnarray}\sum_{n\le\xi}|\mu_{\mathrm{sym}^2\varphi_j}(n)|\ll\xi R^{\epsilon}.\end{eqnarray}
In particular, 
\begin{eqnarray}\frac{1}{L (s,\mathrm{sym}^2\varphi_j)}\ll_{\epsilon}R^{\epsilon},  \textrm{ for }  \mathrm{Re}(s)\ge\sigma_0>1.\end{eqnarray}
Recall that according to [Iw] and [H-L], we have that 
\begin{eqnarray}R^{-\epsilon}\ll_{\epsilon}L (1,\mathrm{sym}^2\varphi_j)\ll_{\epsilon}R^{\epsilon}.\end{eqnarray}
Also from the definition  we have for $f$ fixed \begin{eqnarray}|V_j|^2=|<Op(f)\varphi_j,\varphi_j>|^2\ll 1.\end{eqnarray}

\textbf{Lemma 1}: Given a small $\epsilon_0>0$, there is $\delta_0=\delta_0(\epsilon_0)$, such that for $X=R^{\epsilon}$,
\begin{eqnarray}
&&\sum_{t_j\le R}|\sum_{n=1}^{\infty}\frac{\mu_{\mathrm{sym}^2\varphi_j}(n)}{n}e^{-\frac{n}{X}}-L^{-1}(1, \mathrm{sym}^2\phi_j)|L(1, \mathrm{sym}^2\phi_j)|V_j|^2\nonumber\\
&& \quad \quad \ll_{\epsilon}X^{-\delta_0}R^{1+\epsilon}
\end{eqnarray}

We prove this by dividing  $\varphi_j$'s with $t_j\le R$ into two sets; $G$ those for which $L^{-1}(1, \mathrm{sym}^2\phi_j)$ has no zeros near $\mathrm{Re}(s)=1$ and the rest which we  
denote by $B$. According to the general density theorem [K-M], we can bound $|B|$ as follows.

For $\frac34<\alpha<1$ and $T\ge 1$, let $$N(\varphi_j;\alpha, T)=|\{\rho: L(\rho, \mathrm{sym}^2\phi_j)=0, |\mathrm{Im}(\rho)|\le T, \mathrm{Re}(\rho)\ge \alpha\}|.$$

Theorem 2 in \cite{KM1}  applies to this situation (as in their remark 4, one only needs $\theta<\frac 14$, and this holds since $\theta\le \frac{7}{32}$ according to [K-S], and the proof of [K-M] can be modified directly to 
Maass forms in place of holomorphic ones) and yields:

There are $C_0<\infty$ and $B_0<\infty$ such that,
\begin{eqnarray}\sum_{t_j\le R }N (\varphi_j;\alpha,T)\ll T^{B_0}R^{c_0\frac{1-\alpha}{2\alpha-1}}.\end{eqnarray}
To complete the proof of Lemma 1, we need: 

\textbf{Lemma 2}:  Given $\delta_1>0$ (small), and  $L(s, \mathrm{sym}^2\phi_j)$, ($t_j\le R$) which  has no zeros in  $\mathrm{Re}(s)>1-2\delta_1$ and 
$\mathrm{Im}(s)\le(\log R)^2$, then for $1\le X\le R$,
\begin{eqnarray}
\sum_{n=1}^{\infty}\frac{\mu_{\mathrm{sym}^2\varphi_j}(n)}{n}e^{-\frac{n}{X}}-L^{-1}(1, \mathrm{sym}^2\phi_j)
\ll_{\epsilon} X^{-\delta_1}R^{\epsilon}
\end{eqnarray}

Proof of Lemma 2: We have, 
\begin{eqnarray}
&&\frac{1}{2\pi i}\int_{\mathrm{Re}(s)=2}\Gamma(s)X^sL^{-1}(s+1, \mathrm{sym}^2\phi_j)ds\nonumber\\
&=&\sum_{n=1}^{\infty}\frac{\mu_{\mathrm{sym}^2\varphi_j}(n)}{n}e^{-\frac{n}{X}}
\end{eqnarray}

Now shift the contour integral replacing $\mathrm{Re}(s)=2$ by 
$$\gamma=\gamma_0\cup\gamma_1\cup\gamma_2\cup\gamma_3\cup\gamma_4$$ where 
$\gamma_0$ is the path from $2-i\infty$ to $2-i(\log R)^2$; $\gamma_1$ is a smooth path from $2-i(\log R)^2$ to $-\delta_1-i(\log R)^2/2$; $\gamma_2$ is the path from $-\delta_1-i(\log R)^2/2$ to $-\delta_1+i(\log R)^2/2$; $\gamma_3$ is a smooth path from $-\delta_1+i(\log R)^2/2$ to $2+i(\log R)^2$; $\gamma_4$ is the path from $2+i(\log R)^2$ $2+i\infty$.

We pick up a term from the pole at $s=0$, 
\begin{eqnarray}
&&L^{-1}(1, \mathrm{sym}^2\phi_j)\nonumber\\
&=&\sum_{n=1}^{\infty}\frac{\mu_{\mathrm{sym}^2\varphi_j}(n)}{n}e^{-\frac{n}{X}}+\frac{1}{2\pi i}\int_{\gamma}\Gamma(s)X^sL^{-1}(s+1, \mathrm{sym}^2\phi_j)ds\nonumber
\end{eqnarray}
Now apply the Borel-Caratheodory Theorem as in [T] and our assumptions about  the zeros of $L(s, \mathrm{sym}^2\phi_j)$ to conclude that 
$$\frac{1}{L(s, \mathrm{sym}^2\phi_j)}\ll_{\epsilon} R^{\epsilon} \quad \textrm{along }  \gamma.$$
Hence the integral  $$\int_{\gamma_2}\Gamma(s)X^sL^{-1}(s+1, \mathrm{sym}^2\phi_j)ds\ll R^{\epsilon}X^{-\delta_1}.$$ 
The integrals over the other $\gamma$'s are very small thanks to the $\Gamma$-factor in the integrand and Stirling formula. This proves Lemma 2.

To complete the proof of Lemma 1, let $T=(\log R)^2$ and $\alpha=\alpha(\epsilon_0)$ by sufficiently close to 1, so that (50) yields 
\begin{eqnarray}\sum_{t_j\le R }N (\varphi_j;\alpha,(\log R)^2)\ll R^{\eta_0},\end{eqnarray}
with  $\eta_0<1-\epsilon_0$.

Now let $G$ be the set of those $\varphi_j$'s such that $N (\varphi_j;\alpha,(\log R)^2)=0$ and $B$ the rest. According to (53), $$|B|\ll R^{\eta_0}.$$
Hence from (45), (47) and (48),
\begin{eqnarray}
&&\sum_{t_j\in B}|\sum_{n=1}^{\infty}\frac{\mu_{\mathrm{sym}^2\varphi_j}(n)}{n}e^{-\frac{n}{X}}-L^{-1}(1, \mathrm{sym}^2\phi_j)|L(1, \mathrm{sym}^2\phi_j)|V_j|^2\nonumber\\
&& \quad \quad\quad\ll_{\epsilon} R^{\epsilon}|B|\ll R^{\epsilon}R^{\eta_0}\ll R^{1+\epsilon}X^{-1}. 
\end{eqnarray}
with $X=R^{\epsilon_0}.$

For $t_j\in G$, we have from Lemma 2 that, with $\delta_1=1-\alpha/2,$
\begin{eqnarray}
&&\sum_{t_j\in G}|\sum_{n=1}^{\infty}\frac{\mu_{\mathrm{sym}^2\varphi_j}(n)}{n}e^{-\frac{n}{X}}-L^{-1}(1, \mathrm{sym}^2\phi_j)|L(1, \mathrm{sym}^2\phi_j)|V_j|^2\nonumber\\
&&\quad \quad\ll_{\epsilon} R^{\epsilon}X^{-\delta_1}\sum_{t_j\le R}|V_j|^2L(1, \mathrm{sym}^2\phi_j)\ll R^{1+\epsilon}X^{-\delta_{1}}. 
\end{eqnarray}
On using the weighted version of the main Theorem, namely Proposition's 4 and 5. With (54) and (55), the proof of Lemma 1 is complete. 

Finally, we are ready to remove the weight. From Lemma 1, we have that for $X=R^{\epsilon_0}$,
\begin{eqnarray}
\sum_{t_j\le R}|V_j|^2&=&\sum_{t_j\le R}\sum_{n=1}^{\infty}\frac{\mu_{\mathrm{sym}^2\varphi_j}(n)}{n}e^{-\frac{n}{X}}L(1, \mathrm{sym}^2\phi_j)|V_j|^2\nonumber\\
&&+ O(R^{1+\epsilon}X^{-\delta_{0}}) \nonumber\\
&=& \sum_{n=1}^{\infty} \frac{e^{-\frac{n}{X}}}{n}\sum_{\substack{ab^2c^3=n\\(a,b)=(b,c)=(a,c)=1}}\mu(a)\mu^2(b)\mu(c)\sum_{t_j\le R}\lambda_j(a^2)\lambda_j(b^2)\nonumber\\
&& L(1, \mathrm{sym}^2\phi_j)|V_j|^2+O_{\epsilon}(R^{1+\epsilon}X^{-\delta_{0}}).
\end{eqnarray}

For the inner sum on $t_j$, we consider the following sum for square free $m$,

\begin{eqnarray}
\sum_{t_j\le R}\lambda_j(m^2)L(1, \mathrm{sym}^2\phi_j)|V_j|^2\sim \beta(m) R\nonumber
\end{eqnarray}

By a similar calculation as in Section 4 (p.34-35), we obtain 
 
 \begin{eqnarray}
 \beta(m)&=&\zeta(2)\sum_{a_1b_1c_1=m}\frac{1}{a_1b_1^{3/2}c_1^{1/2}}\sum_{s\geq 1}\frac{\lambda_f(b_1c_1s^2)}{s}\nonumber
\end{eqnarray}

We first evaluate, for $m$ square free,

\begin{eqnarray}
\nu_f(m):&=&\sum_{s\geq1}\frac{\lambda_f(ms^2)}{s}\nonumber\\
& =&\prod_{p\nmid m}B_p(p^{-1})\prod_{p\mid m}\frac{\lambda_f(p)B_p(p^{-1})}{1+p^{-1}}\nonumber
\end{eqnarray}

where $$B_p(x)=\sum_{n\geq 0}\lambda_f(p^{2n})x^n.$$

Thus, the constant after removal of the harmonic weights is 
 \begin{eqnarray}C(f)&=&\zeta(2)\sum_{n\geq 1}\frac 1n\sum_{\substack{ab^2c^3=n\\
(a,b)=(b,c)=(a,c)=1}}\mu(a)\mu^2(b)\mu(c)\sum_{\substack{a_1b_1c_1=ab\\
s\geq 1}}\frac{\lambda_f(b_1c_1s^2)}{a_1b_1^{3/2}c_1^{1/2}s}\nonumber\\
&=&\frac{1}{\zeta(2)}\prod_p\left(1-\frac{\lambda_{f}(p)}{p^{3/2}(1+p^{-1})}\right)\nonumber
 \end{eqnarray}

Hence 
\begin{eqnarray}
\sum_{t_j\le R}|V_j|^2&\sim& C(f)\sum_{t_j\le R}L(1, \mathrm{sym}^2\phi_j)|V_j|^2\nonumber
\end{eqnarray}



Thus, we obtain,
\begin{proposition}
For weight $k$ holomorphic Hecke eigenform $f$,
\begin{eqnarray}
\lim_{R\rightarrow\infty}\frac{1}{R}\sum_{t_j\le R}|<Op(f)\varphi_j,\varphi_j>|^2=C(f)L(\frac12,f)\frac{2^{k-1}|\Gamma(\frac k2)|^2}{\Gamma(k)}.\nonumber
\end{eqnarray}
\end{proposition}

\vspace{.6mm}
\begin{proposition}
Let $\phi(z)$ be an even Maass-Hecke cuspidal
eigenform for  $\Gamma$, with the Laplacian
eigenvalue $\lambda_{\phi}=\frac14+t_{\phi}^2$, we have
\begin{eqnarray}
&&\lim_{R\rightarrow\infty}\frac1R\sum_{t_j\leq R}|<Op(\phi)\varphi_j,\varphi_j>|^2= C(\phi)L(\frac12,\phi)\frac{|\Gamma(\frac14-\frac{it_{\phi}}{2})|^4}{2\pi|\Gamma(\frac12-it_{\psi})|^2}.\nonumber
\end{eqnarray}
\end{proposition}

\vspace{.4cm}

\appendix

\section{A triple product calculation for $\GL_2(\R)$
\\
by Michael Woodbury}

Let $F$ be a number field and $\A=\A_F$ the ring of adeles.  Let $T$ be the subgroup of $\GL_2$ consisting of diagonal matrices with $Z\subseteq T$ the center.  Let $N\subseteq \GL_2$ be the subgroup of upper triangle unipotent matrices so that $P=TN$ is the standard Borel.

Given automorphic representations $\pi_1,\pi_2,\pi_3$ of $\GL_2$ over $F$ such that the product of the central characters is trivial, one can consider the so-called triple product $L$-function $L(s,\Pi)$ attached to $\Pi=\pi_1\otimes\pi_2\otimes\pi_3$, or the completed $L$-function $\Lambda(s,\Pi)$.  This $L$-function is closely related to periods of the form
 $$ I(f) = \int_{[\GL_2]} f_1(g)f_2(g)f(g)dg  $$
where $f=f_1\otimes f_2\otimes f_3$ with $f_i\in \pi_i$, and $[\GL_2]=\A^\times\GL_2(F) \bs GL_2(\A)$.

One example of this relationship arises in the case that $\pi_1$ and $\pi_2$ are cupsidal and $\pi_3$ is an Eisenstein series.  Then $L(s,\Pi)$ is the Rankin-Selberg $L$-function $L(s,\pi_1\times \pi_2)$, and for appropriately chosen $f_3$, the period $I$ gives an integral representation.  Another example occurs when all three representations are cuspidal.  In this case, formulas for $L(s,\Pi)$ have been given by Garrett\cite{Ga}, Gross-Kudla\cite{GK}, Harris-Kudla\cite{HK}, Watson\cite{Wat1} and Ichino\cite{ichino}.

Let us write $\pi_i =\otimes_v \pi_{i,v}$ as a (restricted) tensor product over the places $v$ of $F$, with each $\pi_{i,v}$ an admissible representation of $\GL_2(F_v)$.  Let $\langle \cdot, \cdot \rangle_v$ be a (Hermitian) form on $\pi_i$.  Then, assuming that $f_i = \otimes f_{i,v}$ is factorizable\footnote{As a restricted tensor product, we have chosen vectors $f_{i,v}^0\in \pi_v$ for all but finitely many places $v$.  We require that the local inner forms must satisfy $\langle f_{i,v}^0,f_{i,v}^0\rangle_v=1$ for all such $v$.}, for each $v$ we can consider the matrix coefficient
 $$ I'(f_v) = \int_{\PGL_2(F_v)} \langle \pi_v(g_v)f_{1,v},f_{1,v}\rangle_v
\langle \pi_v(g_v)f_{2,v},f_{2,v}\rangle_v
\langle \pi_v(g_v)f_{3,v},f_{3,v}\rangle_v dg_v, $$ 
and the normalized matrix coefficient
\begin{equation}\label{eq:ItoIprime}
 I_v(f_v) = \zeta_{F_v}(2)^{-2}\frac{L_v(1,\Pi_v,\Ad)}{L_v(1/2,\Pi_v)}I'_v(f_v).
\end{equation}

When each of the representations $\pi_i$ is cuspidal, Ichino proved in \cite{ichino} that there is a constant $C$ such that
\begin{equation}\label{eq:ichino}
 \frac{\abs{I(f)}^2}{\prod_{j=1}^3 \int_{[\GL_2]}\abs{f_j(g)}^2dg} = \frac{C}{2^3}\cdot \zeta_F(2)^2\cdot \frac{\Lambda(1/2,\Pi)}{\Lambda(1,\Pi,\Ad)} \prod_v \frac{I_v(f_v)}{\langle f_v,f_v\rangle_v}
\end{equation}
whenever the denominators are nonzero.  By the choice of normalizations, the product on the right hand side of \eqref{eq:ichino} is in fact a finite product over some number of ``bad'' places.

While Ichino's formula is extremely general, for number theoretic applications it is often important to understand well the bad factors.  For example, subconvexity for the triple product $L$-function as proved by Bernstein-Reznikov in \cite{BR} and Venkatesh in \cite{ven} used, in the former case, Watson's formula from \cite{Wat1} or, in the latter, the result of \cite{W}.

We would like to make \eqref{eq:ichino} more explicit.  First, we remark that the constant $C$ depends only on the choice of measures.  Letting
 $$ K_v = \left\{ \begin{array}{cc} \GL_2(\Z_p) & \mbox{ if $v=p$ is prime}, \\
          \SO(2) & \mbox{ if }v=\infty,\end{array}\right. $$
we choose the local measures $dg_v$ such that the volume of $K_v$ is 1 in all cases, and we choose the global measure on $[\GL_2]$ to be the Tamagawa measure.  With this choice, setting $\Delta_F$ to be the discriminant of $F/\mathbb Q$, we have $C=\frac{1}{\abs{\Delta_F}^{3/2}\zeta_F(2)}$.

Next, we want to replace the \emph{adelic} integrals appearing in \eqref{eq:ichino} with a \emph{classical} version.  It is well-known that if $\varphi_j$ are (classical) modular or Maass forms, then they correspond to automorphic representations $\pi_j$ and $f_j\in \pi_j$.  Although, the correspondence from $\varphi_j$ to $f_j$ is only unique up to a nonzero constant, the choice of constant is irrelevant since \eqref{eq:ichino} is self-normalizing.

If we assume that for each $j=1,2$, $\varphi_j$ is a cuspidal modular or Maass form for the full modular group $\Gamma=\SL_2(\mathbb Z)$, the corresponding $f_j = \otimes f_{j,v}\in \pi_j$ satisfies $f_{j,p}=f_{j,p}^\circ$ for all finite primes, and the difference between integrating over $[\GL_2]$ in the adelic version, and integrating over $X=\SL_2(\Z)\bs \mathbb{H}$ in the classical setting, is the difference between $\vol([\GL_2])=2$ and $\vol(X)$.  So, taking $dA(z)$ to be the probability measure on $X$, we find that \eqref{eq:ichino} yields
\begin{equation}\label{eq:ichinoclassic}
  \frac{\abs{\int_X \varphi_1(z)\varphi_2(z)\varphi_3(z) dA(z)}^2}{\prod_{j=1}^3 \int_{X}\abs{\varphi_j(z)}^2dA(z)} = \frac{1}{2^4}\cdot \frac{\Lambda(1/2,\Pi)}{\Lambda(1,\Pi,\Ad)} \frac{I_\infty(f_\infty)}{\langle f_\infty,f_\infty\rangle_\infty}.
\end{equation}

At the infinite place, $\pi_\infty$ is either a discrete series representation $\DS{k}$ of some weight $k\geq 2$, a limit of discrete series, or it is a principal series $\pi_{it}$ where $\pi_{it} = \Ind_P^G(\abs{\cdot }^{it}\otimes \abs{\cdot}^{-it})$ is obtained as the normalized induction of the character
 $$ \abs{\cdot }^{it}\otimes \abs{\cdot}^{-it}: T(\R)\to \C. $$

Recall that if $f_\infty \in \pi_{it}$ then
 $$ f_\infty(\smatr{u}{0}{0}{u} \smatr{1}{x}{0}{1}\smatr{y}{0}{0}{1} g ) = \abs{y}^{\frac{1}{2}+it}f(g) $$
for all $u,y\in \R^\times$, $x\in \R$ and all $g\in \GL_2(\R)$.  If $\pi=\otimes \pi_v$ corresponds to a Maass form of eigenvalue $\lambda$ under the Laplacian, then $\pi_{\infty}\simeq \pi_{it}$ where $\lambda=\frac{1}{4}+t^2$.  The unitary structure given to $\pi_{it}$ is normalized so as to be given by integration against an invariant probability measure in the circle model.

We now assume that $v\mid \infty$ is a real place.  In this appendix we calculate $I_v$ in the case that $\pi_{1,v}=\DS{k}$ is the discrete series representation of (even) weight $k$, and $\pi_{2,v}=\pi_{it_2}$ and $\pi_{3,v}=\pi_{it_3}$ are principal series representations.

Let
 $$ \SO(2) = \left.\left\{ \kappa_\theta = \matr{\cos{\theta}}{\sin{\theta}}{-\sin{\theta}}{\cos{\theta}} \right\rvert \theta \in \R \right\}. $$
Recall that a function $f_{j,v}\in \pi_{j,v}$ is said to have \emph{weight} $m$ if $f_{j,v}(g\kappa_\theta)=f_{j,v}(g)e^{im\theta}$ for all $g\in \GL_2(\R)$.  As is well known, for each $m\in \Z$, the subspace of $\pi_{j,v}$ consisting of functions of weight $m$ is at most 1-dimensional.

\begin{theorem}\label{appthm}
Let $f_{1,v}\in \DS{k}$ be the vector of weight $k$, let $f_{2,v}\in \pi_{it_2}$ be the vector of weight zero, and let $f_{3,v}\in \pi_{it_3}$ be the vector of weight $-k$ (each normalized\footnote{This normalization ensures that $\langle f_{j,v},f_{j,v}\rangle_v =1$.} so that $f_{j,v}(\smatr{1}{0}{0}{1})=1$.)  Then
\begin{multline}
 I'_v(f_{1,v}\otimes f_{2,v}\otimes f_{3,v}) = \frac{4\pi}{(k-1)!(\frac{1}{2}+it_3)_{\frac{k}{2}}(\frac{1}{2}-it_3)_{\frac{k}{2}}} \\ \times \frac{\Gamma(\frac{k}{2}+it_2+it_3)\Gamma(\frac{k}{2}+it_2-it_3)\Gamma(\frac{k}{2}-it_2-it_3)\Gamma(\frac{k}{2}-it_2+it_3)}{\Gamma(\frac{1}{2}+it_2)\Gamma(\frac{1}{2}-it_2)\Gamma(\frac{1}{2}+it_3)\Gamma(\frac{1}{2}-it_3)}
\end{multline}
and
\begin{equation}\label{eq:appthm}
 I_v(f_{1,v}\otimes f_{2,v}\otimes f_{3,v}) = \frac{2^{k-1}\pi^k}{(\frac{1}{2}+it_3)_{\frac{k}{2}}(\frac{1}{2}-it_3)_{\frac{k}{2}}}.
\end{equation}
where $(z)_m = z(z+1)\cdots(z+m-1)$.
\end{theorem}

\subsection{Real local factors}

For the remainder of this appendix, we work locally over a real place.  Since the place $v$ is assumed fixed, we remove subscripts which refer to it.  In particular, the $L$-functions are local.  We trust that no confusion will arise between these and the global $L$-function considered above.  (For example, $L(s,\Pi)$, to be defined below, represents the local $L$-factor $L_v(s,\Pi)$ appearing in equation~\eqref{eq:ItoIprime}.)

We will assume, however, that the principal series $\pi_{it}$ is unitary.  (This is automatically true if $\pi_{it}$ is the local component of an automorphic representation.)  This implies that $t$ is either real or purely imaginary of absolute value less than $1/2$.  This requirement will be used implicitly to guarantee that certain integrals converge and that certain functions are real valued.  We will use these facts without further mention.

We record the relevant local factors for representations of $\GL_2(\R)$.  Let
 $$ \Gamma_\R(s) = \pi^{-s/2}\Gamma(s/2),\quad \mbox{and}\quad \Gamma_\C(s) = \Gamma_\R(s)\Gamma_\R(s+1)=2(2\pi)^{-s}\Gamma(s) $$
where $\Gamma(s) = \int_0^\infty y^s e^{-y} d^\times y$ when $\re(s)>0$ and is extended by analytic continuation elsewhere.  Note that
\begin{equation}\label{eq:specialvaluesofGamma}
 \Gamma_\R(1) = 1, \qquad \Gamma_\R(2) = \frac{1}{\pi},\qquad \mbox{and} \qquad \Gamma_\C(m) = \frac{(m-1)!}{2^{m-1}\pi^m}.
\end{equation}

We recall basic facts about the local Langlands correspondence for $\GL_2(\R)$ as found in Knapp \cite{Kna}.  The Weil group $W_\R=\C^\times \cup j \C^\times$ where $j^2=-1$ and $jzj^{-1}=\bar{z}$ for $z\in \C^\times$.  The irreducible representations of $W_\R$ are all either 1-dimensional or 2-dimensional.  The 1-dimensional representations are parametrized by $\delta\in \{0,1\}$ and $t\in \C$:
 $$\rho_1(\delta,t): \begin{array}{c} z\mapsto \abs{z}^t \\ j\mapsto (-1)^\delta.\end{array}$$
The irreducible 2-dimensional representations are parametrized by positive integers $m$ and $t\in \C$:
 $$ \rho_2(m,t): \begin{array}{c} re^{i\theta}\mapsto \matr{r^{2t}e^{im\theta}}{0}{0}{r^{2t}e^{-im\theta}} \\ \phantom{.} \\ j\mapsto \matr{0}{(-1)^m}{1}{0} \end{array} $$

Defining $\rho_2(0,t)=\rho_1(0,t)\oplus\rho_1(1,t)$ and $\rho_2(m,t)=\rho_2(\abs{m},t)$, the following is an elementary exercise.

\begin{lemma}\label{lem:LangCorr}
Every (semisimple) finite dimensional representation of $W_\R$ is a direct sum of irreducibles each of dimension one or two.  Under the operations of direct sum and tensor product, the following is a complete set of relations. 
\begin{gather*}
 \rho_2(m,t) \simeq \rho_2(-m,t) \\
 \rho_2(0,t) \simeq \rho_1(0,t)\oplus \rho_1(1,t) \\
 \rho_1(\delta_1,t_1)\otimes \rho_1(\delta_2,t_2) \simeq \rho_1(\delta,t_1+t_2) \\
 \rho_1(\delta,t_1)\otimes \rho_2(m,t_2) \simeq \rho_2(m,t_1+t_2) \\
 \rho_2(m_1,t_1)\otimes \rho_2(m_2,t_2) \simeq \rho_2(m_1+m_2,t_1+t_2)\oplus\rho_2(m_1-m_2,t_1+t_2) 
\end{gather*}
In the third line, $\delta=\delta_1+\delta_2\pmod{2}$.  Moreover, if $\wt{\rho}$ denotes the contragradient of $\rho$ then
 $$ \wt{\rho_1(\delta,t)} \simeq \rho_1(\delta,-t), \quad \mbox{and} \quad \wt{\rho_2(m,t)} \simeq \rho_1(m,-t). $$
\end{lemma}

Attached to each irreducible representation $\rho$ of $W_\R$ is an $L$-factor
 $$ L(s,\rho_1(\delta,t)) = \Gamma_\R(s+t+\delta), \quad \mbox{and} \quad L(s,\rho_2(m,t)) = \Gamma_\C(s+t+\frac{\abs{m}}{2}). $$
Writing a general representation $\rho$ as a direct sum of irreducibles $\rho_1\oplus \cdots \oplus \rho_r$, we define
 $$ L(s,\rho) = \prod_{i=1}^r L(s,\rho_i). $$
In particular, given $\rho$, the adjoint representation is 
 $$ \Ad(\rho) \simeq \rho\otimes \wt{\rho} \ominus \rho_1(0,0) $$
since $\rho_1(0,0)$ is the trivial representation.

Under the Langlands correspondence, admissible representations $\pi$ of $\GL_2(\R)$ correspond to 2-dimensional representations $\rho=\rho(\pi)$ of $W_\R$.   For example, $\rho(\pi_{it})=\rho_1(0,it)\oplus \rho_1(0,-it)$ and $\rho(\DS{k})=\rho_2(k-1,0)$.  Thus the local factors for the discrete series and principal series representations are
 $$ L(s,\DS{k}) = \Gamma_\C(s+(k-1)/2), \quad \mbox{and} \quad
    L(s,\pi_{it}) = \Gamma_\R(s+it)\Gamma_\R(s-it). $$

We define 
 $$ L(s,\Pi) = L(s,\rho(\DS{k})\otimes \rho(\pi_{it_2})\otimes \rho(\pi_{it_3})) $$
and
 $$ L(s,\Pi,\Ad) = L(s,\Ad\rho(\DS{k})\oplus \Ad\rho(\pi_{it_2})\oplus \Ad\rho(\pi_{it_3})). $$

\begin{lemma}\label{lem:Lfactor}
Let $\Pi=\DS{k}\otimes \pi_{it_2}\otimes \pi_{it_3}$.  The normalizing factor relating $I_v$ and $I_v'$ in \eqref{eq:ItoIprime} at a real place $v$ with local factor isomorphic to $\Pi$ is
\begin{multline*}
 \frac{L(1,\Pi,\Ad)}{\Gamma_\R(2)^2L(1/2,\Pi)} =  2^{k-3}\pi^{k-1}(k-1)! \\ \frac{\Gamma(\frac{1}{2}+it_2)\Gamma(\frac{1}{2}-it_2)\Gamma(\frac{1}{2}+it_3)\Gamma(\frac{1}{2}-it_3)}{\Gamma(\frac{k}{2}+it_2+it_3)\Gamma(\frac{k}{2}-it_2+it_3)\Gamma(\frac{k}{2}+it_2-it_3)\Gamma(\frac{k}{2}-it_2-it_3)}.
\end{multline*}
\end{lemma}

\begin{proof}
Using Lemma~\ref{lem:LangCorr}, one can easily show that
\begin{eqnarray*}
 L(1/2,\Pi) = &  \prod_{\varepsilon,\varepsilon'\in \{\pm 1\}} \Gamma_\C\left(\varepsilon it_2+\varepsilon' it_3+\frac{k}{2}\right) \\
  = & 2^{4}(2\pi)^{-2k}\prod_{\varepsilon,\varepsilon'\in \{\pm 1\}} \Gamma\left(\frac{k}{2} + \varepsilon it_2+\varepsilon' it_3\right)
\end{eqnarray*}
and, applying \eqref{eq:specialvaluesofGamma}, $L(1,\Pi,\Ad)$ is equal to
\begin{gather*}
  \big( \Gamma_\C(k) \Gamma_\R(2) \big) \big( \Gamma_\R(1+2it_2) \Gamma_\R(1-2it_2) \Gamma_\R(1) \big)  \big( \Gamma_\R(1+2it_3) \Gamma_\R(1-2it_3) \Gamma_\R(1) \big) \\
  \quad =  \frac{(k-1)!}{2^{k-1}\pi^{k+3}} \Gamma\left( \frac{1}{2}+it_2 \right) \Gamma\left( \frac{1}{2}-it_2 \right) \Gamma\left( \frac{1}{2}+it_3 \right) \Gamma\left( \frac{1}{2}-it_3 \right).
\end{gather*}
Combining these, we arrive at the desired formula.
\end{proof}

\subsection{Whittaker models}

As a matter of notation, set
 $$ a(y) = \matr{y}{0}{0}{1},\quad z(u) = \matr{u}{0}{0}{u}, \quad n(x) = \matr{1}{x}{0}{1}. $$

Let $\pi$ be an irreducible (unitary) infinite dimensional representation of $G$ with central character $\omega$, and let $\psi:\R\to \C^\times$ be a nontrivial additive character.  Then there is a unique space of functions $\W(\pi,\psi)$ isomorphic to $\pi$ such that
\begin{equation}\label{eq:Whitt}
 W(z(u)n(x)g) = \omega(u)\psi(x)W(g)
\end{equation}
for all $g\in G$.  Recall that the inner product on $\W(\pi,\psi)$ is given by
 $$ \langle W,W' \rangle = \int_{\R^\times} W(a(y))\overline{W'(a(y))} d^\times y. $$
We fix $\psi:\R\to \C^\times$ once and for all to be the character $\psi(x)=e^{2\pi i x}$.

If the central character of $\pi$ is trivial, and $W\in \W(\pi,\psi)$ has weight $m$, \eqref{eq:Whitt} becomes
\begin{equation}\label{eq:WhitIP}
  W(z(u)n(x)a(y) \kappa_\theta) = e^{2\pi ix}W(a(y))e^{im\theta}.
\end{equation}
This, by the Iwasawa decomposition, determines $W$ completely provided we can describe $w(y)=W(a(y))$.  This can be accomplished for the weight $k$ vector $W_k^k\in \W(\DS{k},\psi)$ by utilizing the fact that $W_k^k$ is annihilated by the lowering operator $X^- \in \rm{Lie}(\GL_2(\R))$.  Applying $X^-$ to \eqref{eq:WhitIP}, one finds that $w(y)$ satisfies a certain differential equation whose solution is easily obtained.  The unique solution with moderate growth is, up to a constant, 
\begin{equation}\label{eq:DKkWhitvector}
 W_k^k(a(y)) = \left\{ \begin{array}{cl} y^{k/2}e^{-2\pi y} & \mbox{ if }y\geq 0 \\ 0 & \mbox{ if }y<0. \end{array}\right.
\end{equation}

We calculate directly (so long as $\re(s)>1-\frac{k+k'}{2}$) that 
\begin{equation}\label{eq:RSk}
 \begin{array}{rcl}
 \displaystyle{\int_0^\infty W_k^k(a(y)) W_{k'}^{k'}(a(y)) y^{s-1} d^\times y}
  & = & \displaystyle{\int_0^\infty y^{s-1+(k+k')/2}e^{-4\pi y} d^\times y} \\
  & = & \displaystyle{\frac{\Gamma(s-1+(k+k')/2)}{(4\pi)^{s-1+(k+k')/2}}}
 \end{array} 
\end{equation}
By letting $s=1$ and $k=k'$, this implies that
\begin{equation}\label{eq:WkNorm}
 \langle W_k^k,W_k^k \rangle = \frac{(k-1)!}{(4\pi )^k}.
\end{equation}

Analogously, if $W_m^\lambda\in \W(\pi_{it},\psi)$ is a weight $m$-vector which is an eigenvector for the action of the Laplace operator $\Delta$ of eigenvalue $\lambda$, one can apply $\Delta$ to \eqref{eq:Whitt} to see that $w(y)=W_m^\lambda(a(y))$ satisfies the confluent hypergeometric differential equation
\begin{equation}\label{eq:WhittDE}
w''  + \left[ -\frac{1}{4} + \frac{m}{2y} + \frac{\lambda}{y^2}\right] w = 0.
\end{equation}
Therefore, $W_m^\lambda(a(y))= W_{\frac{m}{2},it}(\abs{y})$ is the unique solution of \eqref{eq:WhittDE} with exponential decay as $\abs{y}\to \infty$ and $\lambda = \frac{1}{2}+t^2$.  The weight zero vector $W_0^\lambda$ can be expressed in terms of the incomplete Bessel function:
\begin{equation}\label{eq:PSitWhitvector}
 W_0^\lambda(a(y)) = W_{0,it}(y) = 2\pi^{-1/2}\abs{y}^{1/2} K_{it}(2\pi \abs{y}).
\end{equation}


By formula~(6.8.48) of \cite{transforms}, it follows that
\begin{multline}\label{eq:RSit}
 \int_0^\infty W_{0,it_1}(a(y))W_{0,it_2}(a(y)) y^{s-1} d^\times y \\ = \frac{4}{\pi}\int_0^\infty K_{it_1}(2\pi y) K_{it_2}(2\pi y) y^s d^\times y \\
  = \frac{1}{2\pi^{s+1}}\frac{\Gamma(\frac{s+it_1+it_2}{2})\Gamma(\frac{s-it_1+it_2}{2})\Gamma(\frac{s+it_1-it_2}{2})\Gamma(\frac{s-it_1-it_2}{2})}{\Gamma(s)}.
\end{multline}
Evaluating this at $s=1$ in the case that $t_1=t_2=t$, we have that
\begin{equation}\label{eq:WitNorm}
 \langle W_0^\lambda, W_0^\lambda \rangle = \frac{\Gamma(\frac{1}{2}+it)\Gamma(\frac{1}{2}-it)}{\pi}.
\end{equation}
Note that we have used that $W_0^\lambda(a(y))$ is an even function and $\Gamma(1/2)=\sqrt{\pi}$.

\begin{remark}
An explicit intertwining map $\pi\to \W(\pi,\psi)$ is given, when the integral is convergent, by
\begin{equation}\label{eq:intertwiner}
 f\mapsto W_f\qquad W_f(g) = \pi^{-1/2}\int_\R f(wn(x)g)\overline{\psi(x)}dx
\end{equation}
where $w=\smatr{0}{1}{-1}{0}$, and this can be extended by analytic continuation elsewhere.

As an alternative to the strategy above, one can deduce equations \eqref{eq:RSk} and \eqref{eq:RSit} by working directly from \eqref{eq:intertwiner}.  (See \cite{Ga2}.)  The normalization in \eqref{eq:PSitWhitvector} coincides with this choice of intertwiner.
\end{remark}

\subsection{Proof of Theorem 3}

We are now in a position to prove Theorem~\ref{appthm}.  Having laid the groundwork above, it is a simple consequence of the following result \cite[Lemma~3.4.2]{MV}.

\begin{lemma}[Michel-Venkatesh]\label{lem:MV}
Let $\pi_1,\pi_2,\pi_3$ be tempered representations of $\GL_2(\R)$ with $\pi_3$ a principal series.  Fixing isometries $\pi_1\to \W(\pi_1,\psi)$ and $\pi_2\to \W(\pi_2,\overline{\psi})$, we may associate for $f_j \in \pi_j$ vectors $W_j$ in the Whittaker model.  Then the form $\ell_{\rm RS}: \pi_1\otimes \pi_2 \otimes \pi_3 \to \C$ given by
\begin{multline}\label{eq:RS}
 \ell_{\rm RS} (f_1\otimes f_2\otimes f_3) \\ = \int_{K}\int_{\R^\times} W_1(a(y) \kappa) W_2(a(y) \kappa) f_3(a(y) \kappa) \abs{y}^{-1} d^\times y d\kappa 
\end{multline}
satisfies $\abs{\ell_{\rm RS}}^2=I'(f_1\otimes f_2 \otimes f_3)$.
\end{lemma}
Note that although $\ell_{\rm RS}$ depends on the particular choice of isometry $\pi_j\to \mathcal{W}_j$, the value $\abs{\ell_{\rm RS}}^2$ does not.

For $j=1,2$ we have $\lambda_j=\frac{1}{4}+t_j^2$.  Recall our choice of test functions: $W_1=W_k^k$, $W_2=W_0^{\lambda_2}$, and $f_3\in \pi_{it_3}$ of weight $-k$.  Since the sum of the weights of these is zero, the integral over $K$ in \eqref{eq:RS} is trivial, and 
\begin{align*}
 \ell_{\rm RS}(W_1\otimes W_2\otimes f_3)
 = & \int_0^\infty W_1(a(y))W_2(a(y))f_3(a(y))\abs{y}^{-1} d^\times y \\
 = & \int_0^\infty e^{-2\pi y}y^{k/2} 2\pi^{-1/2} y^{1/2}K_{it_2}(2\pi y) y^{1/2+it_3}y^{-1} d^\times y \\
 = & 2\pi^{-1/2} \int_0^\infty e^{-2\pi y} K_{it_2}(2\pi y) y^{k/2+it_3} d^\times y \\
 = & \frac{2}{(4\pi)^{k/2+it_3}}\frac{\Gamma(\frac{k}{2}+it_2+it_3)\Gamma(\frac{k}{2}-it_2+it_3)}{\Gamma(\frac{1}{2}+\frac{k}{2}+it_3)} 
\end{align*}
In the final line we have used equation~(6.8.28) from \cite{transforms}.  This simplifies further by using the identity $\Gamma(z+m)= \Gamma(z)(z)_m$.

Recall that we have chosen $f_j$ such that $\langle f_j,f_j\rangle = 1$ for each $j$.  Therefore, in order to apply Lemma~\ref{lem:MV}, we must normalize $\ell_{\rm RS}$:
\begin{multline*}
 I'(f_1\otimes f_2\otimes f_3)
 =  \frac{ \abs{\ell_{\rm RS}(W_1\otimes W_2\otimes f_3)}^2}{\langle W_1,W_2\rangle \langle W_2,W_2\rangle}  \\ = \frac{4\pi}{(k-1)!(\frac{1}{2}-it_3)_{\frac{k}{2}}(\frac{1}{2}+it_3)_{\frac{k}{2}}} \times \\ 
 \times\frac{\Gamma(\frac{k}{2}+it_2+it_3)\Gamma(\frac{k}{2}+it_2-it_3)\Gamma(\frac{k}{2}-it_2-it_3)\Gamma(\frac{k}{2}-it_2+it_3)}{\Gamma(\frac{1}{2}+it_2)\Gamma(\frac{1}{2}-it_2)\Gamma(\frac{1}{2}+it_3)\Gamma(\frac{1}{2}-it_3)}
\end{multline*}

To complete the proof, we multiply by the normalizing factor of Lemma~\ref{lem:Lfactor}.

\begin{remark}
If one or more of the representations $\pi_{it_j}$ is a complementary series (i.e.\ if $\lambda_j<\frac{1}{4}$) then the result of Theorem~\ref{appthm} still holds, but the explicit calculation is somewhat different.  In this case, it is no longer true that for $r\in\R$
 $$ \abs{ \Gamma(r+it_j) }^2 = \Gamma(r+it_j)\Gamma(r-it_j), $$
nor is it true that $\langle f_j, f_j\rangle =1$.  Taking into account these differences, however, the final answer ends up agreeing with what has been calculated above.  Alternatively, as explained in \cite{MV}, a suitably polarized version of the main formula is meromorphic in the spectral parameters.  Hence, the result follows by analytic continuation.
\end{remark}




$\mathbf{Acknowledgements.}$ We thank Nalini Anantharaman and Steve Zelditch for clarifying various points in their papers \cite{AZ} and \cite{AZ2}. We are very grateful to the referee for the very careful
readings of the paper and for pointing to numerous errors in the computations and providing corrections.
\vspace{.5cm}
\bibliographystyle{plain}

\begin{thebibliography}{}
\bibitem{AZ} N.  Anantharaman, S. Zelditch, {\it Patterson-Sullivan
distributions and quantum ergodicity}  Ann.
Henri Poincare (2007), no. 2, 361-426.
\bibitem{AZ2}  N. Anantharaman, S.  Zelditch, {\it Intertwining the geodesic flow and the Schršdinger group on hyperbolic surfaces.} Math. Ann. 353 (2012), no. 4, 1103-1156.
\bibitem{Ba} Bateman, {\it Tables of Integral Transforms},
McGraw-Hill, New York, 1954.
\bibitem{BR}
Joseph Bernstein and Andre Reznikov.
\newblock Periods, subconvexity of {$L$}-functions and representation theory.
\newblock {\em J. Differential Geom.}, 70(1):129--141, 2005.
\bibitem{Bu} D. Bump, {\it Automorphic Forms on $GL(3,R)$},
Lecture Notes in Mathematics, Vol. 1083.
\bibitem{DI} J.-M. Deshouillers and H. Iwaniec, {\it Kloosterman Sums
and Fourier Coefficients of Cusp
Forms,} Invent. Math. 70, 219-288, 1982.
\bibitem{ef} B. Eckhardt, Fishman et al. {\it Approach to ergodicity
in quantum wave
functions}, Phys. Rev. E 52 (1995)
\bibitem{transforms}
A.~Erd{\'e}lyi, W.~Magnus, F.~Oberhettinger, and F.~G. Tricomi.
\newblock {\em Tables of integral transforms. {V}ol. {I}}.
\newblock McGraw-Hill Book Company, Inc., New York-Toronto-London, 1954.
\newblock Based, in part, on notes left by Harry Bateman.
\bibitem{FP} M. Feingold, A. Peres, {\it Distributin of matrix
elements of chaotic
systems}, Phy. Rev. A 34 (1986).
\bibitem{Ga2}
Paul~B. Garrett.
\newblock Standard archimedean integrals for {$GL(2)$}.
\newblock available at http://www.math.umn.edu/$\sim$garrett/.
\bibitem{Ga}
Paul~B. Garrett.
\newblock Decomposition of {E}isenstein series: {R}ankin triple products.
\newblock {\em Ann. of Math. (2)}, 125(2):209--235, 1987.
\bibitem{GJ}
Roger Godement and Herv{\'e} Jacquet.
\newblock {\em Zeta functions of simple algebras}.
\newblock Lecture Notes in Mathematics, Vol. 260. Springer-Verlag, Berlin,
  1972.
\bibitem{GRS} Amit Ghosh, Andre Reznikov, Peter Sarnak, {\it Nodal Domains of Maass Forms I}, GAFA 2013, Vol 23, pp1515-1568. 
\bibitem{gr} I. S. Gradshteyn and I. M. Ryzhik, {\it
Table of integrals, series, and products}, 4th ed., Academic Press,
New York, 1965.
\bibitem{GK}
Benedict~H. Gross and Stephen~S. Kudla.
\newblock Heights and the central critical values of triple product
  {$L$}-functions.
\newblock {\em Compositio Math.}, 81(2):143--209, 1992.
\bibitem{HK} M. Harris, S. Kudla,
{\it The Central Critical Value of a Triple Product L-Function}, Ann.
of Math. (2) V. 133, 605-672.
\bibitem{He} D. Hejhal, The Selberg trace formula for PSL(2,R). Vol.
I, Lecture Notes in Mathematics, (1976), Vol. 548
\bibitem{hl} J. Hoffstein and P. Lockhart, {\it
Coefficients of Maass forms and the Siegel zero}, Ann. of Math.(2)
{\bf 140} (1994), no. 1, 161--181.
\bibitem{ichino} A. Ichino, {\it Trilinear forms and the central
values of triple product $L$-functions},
Duke Math. J. Volume 145, Number 2 (2008), 281-307.
\bibitem{II} A. Ichino, T. Ikeda, {\it On the periods of automorphic
forms on special orthogonal groups and the Gross-Prasad conjecture}.
GAFA, V.19, 2010, 1378-1425.
\bibitem{Iw90} H. Iwaniec, {\it Small eigenvalues of Laplacian for $\Gamma_0(N)$,} Acta Arith. 16 (1990), 65-82.
\bibitem{iwa} H. Iwaniec, {\it Topics in classical automorphic
forms}, Graduate Studies in Mathematics, vol. 17, American
Mathematical Society, Rhode Island, 1997.
\bibitem{ILS} H. Iwaniec, W. Luo, and, P. Sarnak, {\it Low lying zeros
of families of L-functions}, Publ. IHES, 2000
\bibitem{IK} H. Iwaniec, E. Kowalski, {\it  Analytic number theory.}
American Mathematical Society Colloquium Publications, 53. American
Mathematical Society, Providence, RI, 2004.
\bibitem{ILS} Iwaniec, H.; Luo, W.; Sarnak, P. {\it Low lying zeros of
families of $L$-functions.}  Inst. Hautes eudes Sci. Publ. Math.  No.
91  (2000), 55--131 (2001)
\bibitem{J94} D. Jakobson, {\it QUE for Eisenstein Series on $PSL_2(Z)\backslash
PSL_2(R)$}, Annales de l'Institut Fourier, 44(5) (1994), 1477-1504.
\bibitem{J97} D. Jakobson, {\it Equidistribution of cusp forms on
$PSL_2(Z)\backslash
PSL_2(R)$}, Ann. Inst. Fourier, 1997.
\bibitem{jut} Jutila, M. {\it
The spectral mean square of Hecke $L$-functions on the critical
line.} Publ. Inst. Math. (Beograd) (N.S.) 76(90) (2004), 41--55.
\bibitem{KimS} H. Kim and P. Sarnak, {\it Appendix to Functoriality for the exterior square of GL4
and symmetric fourth of GL2} by H. Kim. JOURNAL OF THEAMERICAN MATHEMATICAL SOCIETY, Volume 16, Number 1, Pages 139Ð183.

\bibitem{Kna}
A.~W. Knapp.
\newblock Local {L}anglands correspondence: the {A}rchimedean case.
\newblock In {\em Motives ({S}eattle, {WA}, 1991)}, volume~55 of {\em Proc.
  Sympos. Pure Math.}, pages 393--410. Amer. Math. Soc., Providence, RI, 1994.
\bibitem{KM} E. Kowalski and P. Michel, {\it The analytic rank of
$J_0(q)$ and zeros of automorphic $L$-functions}, Duke Math. 1999.
\bibitem{KM1} E. Kowalski and P. Michel, {\it Zeors of families of automorphic $L$-functions close to 1}, Pacific J. of Math., Vol. 207, No. 2, 2002.

\bibitem{Ku} N. V. Kuznetsov, {\it Petersson's Conjecture for Cusp
Forms of Weight Zero and Linnik's
Conjecture.} Math. USSR Sbornik 29 1981, 299-342.
\bibitem{Lindenstrauss}  E. Lindenstrauss, {\it Invariant measures and
arithmetic quantum unique ergodicity.} Ann. of Math. (2) 163 (2006),
no. 1,
165-219.
\bibitem{Lang} S. Lang, {\it $SL_2(\mathbb{R})$}, GTM, 105. 1985.
\bibitem{LY} J. Liu and Y. Ye, {\it Subconvexity for Rankin-Selberg $L$-functions of Maass Forms}. GAFA, 2002.

\bibitem{Luo0}W. Luo, Z. Rudnick and P. Sarnak, {\it The variance of
arithmetic measures associated to closed geodesics on the modular
surface},
ArXiv.
\bibitem{Luo} W. Luo, {\it Zeros of Hecke $L$-functions associated
with cusp forms.} Acta Arith. 71 (1995), no. 2, 139--158.
\bibitem{Luo6} W. Luo, {\it Values of symmetric square $L$-functions at
$1$,} J. Reine Angew. Math. 506 (1999) 215-235.
\bibitem{Luoweyl} W. Luo, {\it Nonvanishing of L-Values and the Weyl
Law}. Ann. of Math. 2001.
\bibitem{Luo7} W. Luo and P. Sarnak, {\it Quantum ergodicity of
eigenfunctions on $PSL_2(Z)\setminus
H$}, Inst. Hautes Etudes Sci. Publ. Math. No. 81 (1995), 207--237.
\bibitem{Luo2} W. Luo and P. Sarnak, {\it Quantum variance for Hecke
eigenforms}, 2004 Ann. Sci. Ecole Norm. Sup. (4) 37 (2004), no. 5,
769--799.
\bibitem{MV}
P.~Michel and A.~Venkatesh.
\newblock The subconvexity problem for $gl(2)$.
\newblock {\em Publications Math IHES}, 2009.
\bibitem{PSR} I. Piatetski-Shapiro and S. Rallis, {\it Rankin Triple
$L$-functions}, Compositio Math. 1987, 31-115.
\bibitem{Ra} M. Ratner, {\it The rate of mixing for geodesic and
horocycle flows,} Ergodic Theory Dynam. Sys. 7 (1987) 267-288.
\bibitem{c} M. Ratner, {\it The central limit theorem for geodesic
flows on $n$-dimensional manifolds of negative curvature} Israel J.
Math. 16 (1973) 181-197.
\bibitem{RS} Z. Rudnick and P. Sarnak, {\it The behavior of
eigenstates of hyperbolic manifolds, }Com. Math. Phys. 1994
\bibitem{Sarnak0}  P. Sarnak, {\it Arithmetic Quantum Chaos},
The Schur Lectures, 2003.
\bibitem{Sarnak1}  P. Sarnak, {\it Estimates for Rankin-Selberg
$L$-functions and
quantum unique ergodicity.} J. Funct. Anal. 184 (2001), no. 2,
419--453.
\bibitem{Sel} A. Selberg, Collected Papers, Vol. 1.
\bibitem{sou} K. Soundararajan {\it Quantum unique
ergodicity for $SL_2(Z)\setminus\mathbb{H}$}, Annals of Math. 2010.
\bibitem{ti} E. C. Titchmarsh, {\it The theory of the
Riemann zeta-function}, 2nd ed., Oxford University Press, 1986.
\bibitem{ven}
Akshay Venkatesh.
\newblock Sparse equidistribution problems, period bounds and subconvexity.
\newblock {\em Ann. of Math. (2)}, 172(2):989--1094, 2010.
\bibitem{Wat1} T. Watson, {\it Central Value of Rankin Triple
L-function for Unramified Maass Cusp Forms,} Princeton thesis, 2004.
\bibitem{W}
Michael Woodbury.
\newblock Trilinear forms and subconvexity of the triple product
  {$L$}-function.
\newblock available at http://www.math.columbia.edu/{$\sim$}woodbury/research/.
\bibitem{Zelditch} S. Zelditch, 
{\it Mean Lindelof hypothesis and equidistribution of cusp forms and Eisenstein series.} 
J. Funct. Anal. 97 (1991), no. 1, 1-49.
\bibitem{Zel} S. Zelditch, {\it On the rate of quantum
ergodicity}, Comm. Math. Phys. 160 (1994)81-92.
\bibitem{Ze} S. Zelditch, {\it Uniform distribution of eigenfunctions
on compact hyperbolic
surfaces}, Duke Math. J. 55(1987).
\bibitem{Z} P. Zhao, {\it Quantum Variance of Maass-Hecke Cusp Forms},
CMP, 2010.
\end{thebibliography}

\end{document}